\documentclass[12pt]{amsart}

\newtheorem{theorem}{Theorem}[section]

\newtheorem{lemma}[theorem]{Lemma}
\newtheorem{proposition}[theorem]{Proposition}
\newtheorem{corollary}[theorem]{Corollary}

\theoremstyle{remark}
\newtheorem{definition}[theorem]{Definition}

\newtheorem*{remark}{Remark}
\newtheorem*{lastremark}{Final remark}

\newcommand{\bC}{\mathbb{C}}

\newcommand{\bQ}{\mathbb{Q}}
\newcommand{\bR}{\mathbb{R}}
\newcommand{\bZ}{\mathbb{Z}}

\newcommand{\cE}{{\mathcal{E}}}

\newcommand{\cO}{{\mathcal{O}}}
\newcommand{\cP}{{\mathcal{P}}}

\newcommand{\cW}{{\mathcal{W}}}

\newcommand{\adj}{\mathrm{Adj}}
\newcommand{\alphabar}{\bar{\alpha}}

\newcommand{\disc}{\mathrm{disc}}
\newcommand{\disp}{\displaystyle}
\newcommand{\et}{\quad\mbox{and}\quad}

\newcommand{\GL}{\mathrm{GL}}
\newcommand{\lambdahat}{\hat{\lambda}}

\newcommand{\matrice}[4]{\begin{pmatrix}#1&#2\\#3&#4\end{pmatrix}}

\newcommand{\SL}{\mathrm{SL}}
\newcommand{\trace}{\mathrm{tr}}
\newcommand{\trans}{{\,{}^t\hskip-1pt}}
\newcommand{\tM}{{\,{}^t\hskip-1pt M}}
\newcommand{\tS}{{\,{}^t\hskip-1pt S}}
\newcommand{\tux}{{\,{}^t\ux}}
\newcommand{\tW}{{\,{}^t\hskip-0.3pt W}}
\newcommand{\ua}{\mathbf{a}}
\newcommand{\ub}{\mathbf{b}}
\newcommand{\um}{\mathbf{m}}
\newcommand{\un}{\mathbf{n}}
\newcommand{\up}{\mathbf{p}}
\newcommand{\uv}{\mathbf{v}}
\newcommand{\uw}{\mathbf{w}}
\newcommand{\ux}{\mathbf{x}}
\newcommand{\uy}{\mathbf{y}}

\newcommand{\Res}{\mathrm{Res}}
\newcommand{\where}{\quad\mbox{where}\quad}

\newcommand{\acc}{
\raisebox{-4pt}{\rule{.4pt}{4pt}} \rule{110pt}{0.4pt}
\rule{.4pt}{4pt} \rule{110pt}{0.4pt}
\raisebox{-4pt}{\rule{.4pt}{4pt}} }

\newcommand{\acci}{
\raisebox{-4pt}{\rule{.4pt}{4pt}} \rule{80pt}{0.4pt}
\rule{.4pt}{4pt} \rule{80pt}{0.4pt}
\raisebox{-4pt}{\rule{.4pt}{4pt}} }

\newcommand{\accii}{
\raisebox{-4pt}{\rule{.4pt}{4pt}} \rule{40pt}{0.4pt}
\rule{.4pt}{4pt} \rule{40pt}{0.4pt}
\raisebox{-4pt}{\rule{.4pt}{4pt}} }

\newcommand{\acciii}{
\raisebox{-4pt}{\rule{.4pt}{4pt}} \rule{20pt}{0.4pt}
\rule{.4pt}{4pt} \rule{20pt}{0.4pt}
\raisebox{-4pt}{\rule{.4pt}{4pt}} }

\newcommand{\lacc}{
\raisebox{-4pt}{\rule{.4pt}{4pt}} \rule{100pt}{0.4pt}
\rule{.4pt}{4pt} }

\newcommand{\lacci}{
\raisebox{-4pt}{\rule{.4pt}{4pt}} \rule{80pt}{0.4pt}
\rule{.4pt}{4pt} }

\newcommand{\laccii}{
\raisebox{-4pt}{\rule{.4pt}{4pt}} \rule{40pt}{0.4pt}
\rule{.4pt}{4pt} }

\newcommand{\racci}{
\rule{.4pt}{4pt} \rule{80pt}{0.4pt}
\raisebox{-4pt}{\rule{.4pt}{4pt}} }

\newcommand{\raccii}{
\rule{.4pt}{4pt} \rule{40pt}{0.4pt}
\raisebox{-4pt}{\rule{.4pt}{4pt}} }

\newcommand{\lacciii}{
\raisebox{-4pt}{\rule{.4pt}{4pt}} \rule{20pt}{0.4pt}
\rule{.4pt}{4pt} }

\newcommand{\racciii}{
\rule{.4pt}{4pt} \rule{20pt}{0.4pt}
\raisebox{-4pt}{\rule{.4pt}{4pt}} }

\newcommand{\g}{g}

\begin{document}

\baselineskip=16pt 

\title[Markoff-Lagrange spectrum and extremal numbers]
{Markoff-Lagrange spectrum and extremal numbers}
\author{Damien ROY}
\address{
   D\'epartement de Math\'ematiques\\
   Universit\'e d'Ottawa\\
   585 King Edward\\
   Ottawa, Ontario K1N 6N5, Canada}
\email{droy@uottawa.ca}
\subjclass[2000]{Primary 11J06; Secondary 11J13}
\keywords{Lagrange spectrum, Markoff spectrum, simultaneous
Diophantine approximation, extremal numbers.}
\thanks{Work partially supported by NSERC and CRM}

\begin{abstract}
Let $\gamma = (1+\sqrt{5})/2$ denote the golden ratio.  H.~Davenport
and W.~M.~Schmidt showed in 1969 that, for each non-quadratic
irrational real number $\xi$, there exists a constant $c>0$ with the
property that, for arbitrarily large values of $X$, the inequalities
\[
 |x_0|\le X, \quad
 |x_0\xi-x_1| \le cX^{-1/\gamma}, \quad
 |x_0\xi^2-x_2| \le cX^{-1/\gamma}
\]
admit no non-zero solution $(x_0,x_1,x_2)\in\bZ^3$.  Their result is
best possible in the sense that, conversely, there are countably
many non-quadratic irrational real numbers $\xi$ such that, for a
larger value of $c$, the same inequalities admit a non-zero integer
solution for each $X\ge 1$.  Such \emph{extremal} numbers are
transcendental and their set is stable under the action of
$\GL_2(\bZ)$ on $\bR\setminus\bQ$ by linear fractional
transformations.  In this paper, it is shown that there exists
extremal numbers $\xi$ for which the Lagrange constant $\nu(\xi) =
\liminf_{q\to\infty} q\,\|q\xi\|$ is $1/3$, the largest possible
value for a non-quadratic number, and that there is a natural
bijection between the $\GL_2(\bZ)$-equivalence classes of such
numbers and the non-trivial solutions of Markoff's equation.
\end{abstract}

\maketitle

\section{Introduction}
 \label{sec:intro}

The purpose of this paper is to present a link between two
relatively distant topics of Diophantine approximation.  The first
one concerns the \emph{Lagrange constant} $\nu(\xi)$ of a real
number $\xi$ defined as the infimum of all real numbers $c>0$ for
which the inequality
\[
 \Big| \xi-\frac{p}{q} \Big| \le \frac{c}{q^2}
\]
has infinitely many solutions $(p,q)\in\bZ^2$ with $q\ge 1$.  This
constant, which vanishes when $\xi\in \bQ$, provides a measure of
approximation of $\xi$ by rational numbers. It is also given by
\[
 \nu(\xi) = \liminf_{q\to\infty} q \|q\xi\|,
\]
where $\|x\|$ stands for the distance from a real number $x$ to a
closest integer.  The \emph{Lagrange spectrum} is the set $\nu(\bR)$
of values of $\nu$.  It is a subset of the interval $[0,1/\gamma]$
where $\gamma=(1+\sqrt{5})/2$ denotes the golden ratio. Thanks to
work of Markoff, the portion of the spectrum in the subinterval
$(1/3,1/\gamma]$ is well understood (see \cite[Ch.~II, \S6]{Ca}). It
forms a countable discrete subset of this subinterval with $1/3$ as
its only accumulation point.  Moreover the real numbers $\xi$ for
which $\nu(\xi)>1/3$ are all quadratic. As a consequence, any
transcendental real number $\xi$ has $\nu(\xi)\le 1/3$.  In the
range $[0,1/3]$, the situation becomes more complicated.  Although,
with respect to Lebesgue measure, almost all real numbers $\xi$ have
$\nu(\xi)=0$, we know in particular that there are uncountably many
$\xi\in\bR$ with $\nu(\xi)=1/3$.

The second topic is the problem of simultaneous rational
approximations to a real number and its square, from a uniform
perspective. In 1969, H.~Davenport and W.~M.~Schmidt showed
\cite[Thm.~1a]{DS} that, for each non-quadratic irrational real
number $\xi$, there exists a constant $c>0$ with the property that,
for arbitrarily large values of $X$, the inequalities
\[
 |x_0|\le X, \quad
 |x_0\xi-x_1| \le cX^{-1/\gamma}, \quad
 |x_0\xi^2-x_2| \le cX^{-1/\gamma}
\]
admit no non-zero solution $(x_0,x_1,x_2)\in\bZ^3$. Recently, it was
established \cite[Thm.~1.1]{RcubicI} that their result is best
possible in the sense that, conversely, there are countably many
non-quadratic irrational real numbers $\xi$ which we henceforth call
\emph{extremal} such that, for a larger value of $c$, the same
inequalities admit a non-zero integer solution for each $X\ge 1$.
Our objective here is to show the existence of extremal numbers
$\xi$ with $\nu(\xi)=1/3$ and to show how this set is intimately
linked with Markoff's theory.

In the next section, we present the main results of Markoff's theory
from a point of view pertaining to the study of extremal numbers.
Then, in Section 3, we construct a family of extremal numbers
$\xi_\um$ parametrized by all solutions in positive integers
$\um=(m, m_1,m_2)$ of the Markoff equation
\begin{equation}
 \label{Markoffeq}
 m^2 + m_1^2 + m_2^2 = 3 m m_1 m_2,
\end{equation}
up to permutation, except $\um=(1,1,1)$. Our main result is that
these numbers $\xi_\um$ constitute a system of representatives of
the equivalence classes of extremal numbers $\xi$ with
$\nu(\xi)=1/3$, under the action of $\GL_2(\bZ)$ on
$\bR\setminus\bQ$ by linear fractional transformations.  To prove
this, we develop further the properties of approximation to extremal
numbers by quadratic real numbers obtained in \cite[\S8]{RcubicI}.
Each extremal number $\xi$ comes with a sequence of best quadratic
approximations $(\alpha_i)_{i\ge 1}$ which is uniquely determined by
$\xi$ up to its first terms.  In Section 4, we show that the
sequence of their conjugates $(\alphabar_i)_{i\ge 1}$ admits exactly
two accumulation points $\xi'$ and $\xi''$ which are also extremal
numbers and which we call the \emph{conjugates} of $\xi$.  Then, in
Section 5, we show that $\nu(\xi) = \nu(\xi') = \nu(\xi'')$ and that
these Lagrange constants can be computed as the infimums of the
absolute values of the binary real quadratic forms
\[
 |\xi-\xi'|^{-1}(T-\xi U)(T-\xi' U)
 \et
 |\xi-\xi''|^{-1}(T-\xi U)(T-\xi'' U)
\]
on $\bZ^2\setminus\{(0,0)\}$. The latter quantities admit handy
representations in terms of doubly infinite words attached to the
continued fraction expansions of $\xi$ and $\xi'$ on one hand, and
of $\xi$ and $\xi''$ on the other hand. This is at the basis of
Markoff's original approach. However, it requests that $0<\xi<1$ and
$\max\{\xi',\xi''\} < -1$. In Section 6, we show that each extremal
number is $\GL_2(\bZ)$-equivalent to exactly one extremal number
$\xi$ with these properties and with conjugates $\xi'$ and $\xi''$
of different integral parts.  We say that such an extremal number is
\emph{balanced}. We also provide a characterization of the numbers
$\xi_\um$ in terms of their continued fraction expansions. Finally,
we conclude in Section 7 with the proof of our main result by
showing that any balanced extremal number $\xi$ with $\nu(\xi)=1/3$
is equivalent to some $\xi_\um$ on the basis of the strong
combinatorial properties shared by the two doubly infinite words
attached to $\xi$. As a corollary, we obtain that an extremal number
$\xi$ has $\nu(\xi) = 1/3$ if and only if its sequence of best
quadratic approximations $(\alpha)_{i\ge 1}$ satisfies
$\nu(\alpha_i)>1/3$ for infinitely many indices $i$.

%
%

\section{Markoff's theory}
\label{sec:M}

A general reference for this section is the exposition given by
J.~W.~S.~Cassels in Chapter II of \cite{Ca}.  In the presentation
below, we reinterpret his constructions in \cite[Ch.~II, \S3]{Ca},
from a point of view closer to the approach of H.~Cohn in \cite{Co},
to align them with similar constructions arising from the study of
extremal numbers.

We recall first that the group $\GL_2(\bQ)$ acts on the set
$\bR\setminus\bQ$ of irrational real numbers by
\begin{equation}
 \label{M:eq:actionGL}
 \g\cdot\xi = \frac{a\xi+b}{c\xi+d}
 \quad\text{if}\quad
 \g=\begin{pmatrix}a &b\\c &d\end{pmatrix} \in \GL_2(\bQ),
\end{equation}
and that we have $\nu(\g\cdot\xi) = \nu(\xi)$ for any
$\g\in\GL_2(\bZ)$ and any $\xi \in \bR \setminus \bQ$ \cite[Ch.~I,
\S3, Cor.]{Ca}.  Consequently, the Lagrange spectrum can be
described as the set of values taken by $\nu$ on a set of
representatives of the equivalence classes of $\bR\setminus\bQ$
under $\GL_2(\bZ)$.

A real binary quadratic form $F(U,T)=rU^2+qUT+sT^2 \in \bR[U,T]$ is
said to be \emph{indefinite} if its \emph{discriminant}
$\disc(F)=q^2-4pr$ is positive.  For such a form, one is interested
in the quantity
\[
 \mu(F)
 := \inf\{\,|F(x,y)|\,;\, (x,y)\in\bZ^2,\ (x,y)\neq (0,0)\,\}.
\]
Keeping the same notation as in \eqref{M:eq:actionGL}, the group $\bR^*
\times \GL_2(\bZ)$ acts on the set of real indefinite binary
quadratic forms by
\[
 (\lambda,\g)\cdot F(U,T) = \lambda F((U,T)\g)
    = \lambda F(aU+cT,\,bU+dT),
\]
and this action fixes the ratio $\mu(F)/\sqrt{\disc(F)}$.  The
\emph{Markoff spectrum} is the set of values of these quotients
$\mu(F) / \sqrt{\disc(F)}$ where $F$ runs through the set of all
real indefinite binary quadratic forms or equivalently through a
system of representatives of the equivalence classes of these forms
under the above action of $\bR^* \times \GL_2(\bZ)$.  Although this
spectrum contains strictly the Lagrange spectrum \cite[Ch.~3,
Thm.~1]{CF}, a remarkable feature of Markoff's theory is that the
trace of the two spectra in the interval $(1/3,1/\gamma]$ are the
same (recall that $\gamma=(1+\sqrt{5})/2$).

The theory provides explicit sets of representatives both for the
equivalence classes of real numbers $\xi$ with $\nu(\xi)>1/3$ and
for the equivalence classes of real indefinite binary quadratic
forms $F$ with $\mu(F)/\sqrt{\disc(F)} > 1/3$. They are
parameterized by the solutions in positive integers $\um = (m,\,
m_1,\, m_2)$ of Markoff's equation \eqref{Markoffeq} upon
identifying two solutions when one is a permutation of the other.
Setting aside the ``degenerate solutions'' $(1,1,1)$ and $(2,1,1)$
which have at least two equal entries, all other solutions in
positive integers appear once and only once in the rooted binary
tree
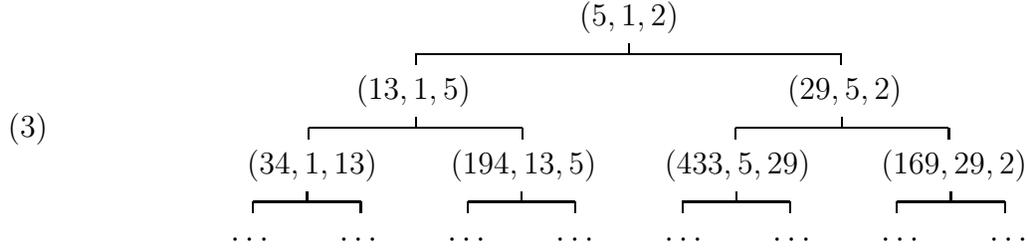
\begin{equation}
 \label{M:Mtree}
  \begin{matrix}
  (5,1,2)\\[-3pt]
  \acci\\[3pt]
  (13,1,5)  \hspace{120pt} (29,5,2)\\[-3pt]
  \accii \hspace{80pt} \accii\\[3pt]
  \hspace{6pt} (34,1,13) \hspace{28pt} (194,13,5) \hspace{26pt}
  (433,5,29) \hspace{27pt} (169,29,2)\\[-3pt]
  \acciii \hspace{40pt} \acciii \hspace{40pt}
  \acciii \hspace{40pt} \acciii\\[3pt]
  \cdots \hspace{25pt} \cdots \hspace{25pt} \cdots \hspace{25pt}
  \cdots \hspace{25pt}
  \cdots \hspace{25pt} \cdots \hspace{25pt} \cdots \hspace{25pt}
  \cdots
  \end{matrix}
\end{equation}
where each node $(m,\, m_1,\, m_2)$ has successors given by
$(3mm_1-m_2,\, m_1,\, m)$ on the left and by $(3mm_2-m_1,\,m,\,m_2)$
on the right.  Moreover, all nodes $(m,\, m_1,\, m_2)$ satisfy $m>
\max\{m_1,m_2\}$ \cite[Ch.~II, \S2]{Ca}.

The same construction starting with $(2,1,1)$ as a root provides a
tree which contains exactly once each triple of positive integers
$(m,\, m_1,\, m_2)$ satisfying \eqref{Markoffeq} and $m>
\max\{m_1,m_2\}$.  In this new tree, each non-degenerate solution is
duplicated, with the tree \eqref{M:Mtree} appearing as its left
half.  This suggests to extend the latter by adding $(2,1,1)$ as a
right ancestor of $(5,1,2)$:
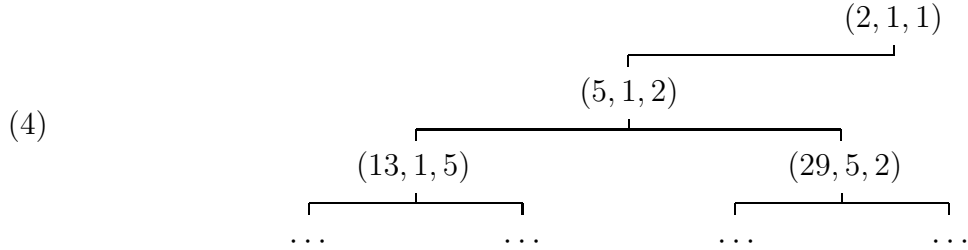
\begin{equation}
 \label{M:Mtreex}
  \begin{matrix}
  \hspace{200pt} (2,1,1)\\[-3pt]
  \hspace{100pt}\lacc\\[3pt]
  (5,1,2)\\[-3pt]
  \acci\\[3pt]
  (13,1,5)  \hspace{120pt} (29,5,2)\\[-3pt]
  \accii \hspace{80pt} \accii\\[3pt]
  \cdots \hspace{65pt} \cdots \hspace{65pt}
  \cdots \hspace{65pt} \cdots\\[-3pt]
  \end{matrix}
\end{equation}
In this extended tree, a node $\um=(m,m_1,m_2)$ has $m_1>m_2$ if and
only if $\um$ has a left ancestor.  In the sequel, we denote by
$\Sigma^*$ the set of all nodes of the tree \eqref{M:Mtreex}, and by
$\Sigma = \Sigma^* \cup \{(1,1,1)\}$ the set of all solutions of the
Markoff equation \eqref{Markoffeq}.

The next proposition lifts \eqref{M:Mtree} to a tree whose nodes are
triples of symmetric matrices in $\SL_2(\bZ)$ (compare with
\cite[Ch.~II, \S3]{Ca} and \cite[\S5]{Co}).

\begin{proposition}
 \label{M:propStree}
Put $M=\matrice{3}{1}{-1}{0}$ and consider the binary rooted tree
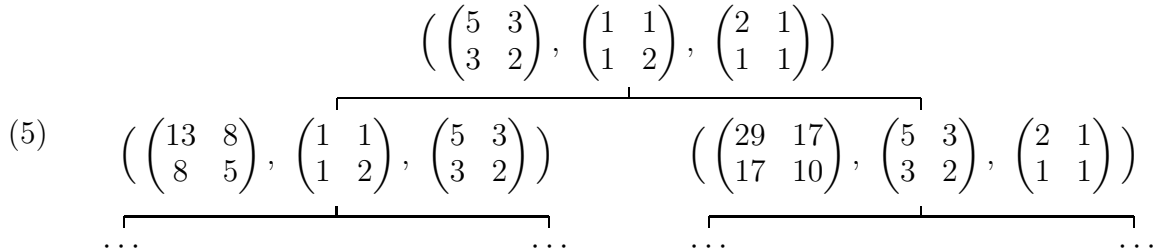
\begin{equation}
  \label{M:propStree:tree}
  \begin{matrix}
  \Big(\matrice{5}{3}{3}{2},\ \matrice{1}{1}{1}{2},\
       \matrice{2}{1}{1}{1}\Big)\\[-3pt]
  \acc\\[3pt]
  \Big(\matrice{13}{8}{8}{5},\ \matrice{1}{1}{1}{2},\
       \matrice{5}{3}{3}{2}\Big)  \hspace{50pt}
  \Big(\matrice{29}{17}{17}{10},\ \matrice{5}{3}{3}{2},\
       \matrice{2}{1}{1}{1}\Big)\\[-1pt]
  \acci \hspace{60pt} \acci\\
  \cdots \hspace{146pt} \cdots \hspace{44pt} \cdots
         \hspace{146pt} \cdots
  \end{matrix}
\end{equation}
where the successors of each node $(\ux,\ux_1,\ux_2)$ are $(\ux_1
M\ux,\ux_1,\ux)$ on the left and $(\ux M\ux_2,\ux,\ux_2)$ on the
right.  Then each node $(\ux,\ux_1,\ux_2)$ of this tree is a triple
of symmetric matrices in $\SL_2(\bZ)$ with positive entries of the
form
\begin{equation}
 \label{M:propStree:eq2}
  \ux=\matrice{m}{k}{k}{\ell}, \quad
  \ux_1=\matrice{m_1}{k_1}{k_1}{\ell_1},\quad
  \ux_2=\matrice{m_2}{k_2}{k_2}{\ell_2}
\end{equation}
satisfying both $\ux = \ux_1 M\ux_2$ and $\max\{k,\ell\}\le m\le
2k$. Moreover, the tree formed by replacing each of these triples of
matrices $(\ux,\ux_1,\ux_2)$ by the triple of their upper left
entries $(m,m_1,m_2)$ is exactly the tree \eqref{M:Mtree} of
non-degenerate solutions of the Markoff equation.
\end{proposition}

\begin{proof}
We first note that the triple of upper left entries of the root
of this tree is the root $(5,1,2)$ of the Markoff tree
\eqref{M:Mtree}. Now, suppose that a node $(\ux,\ux_1,\ux_2)$
of the tree consists of symmetric matrices in $\SL_2(\bZ)$
satisfying $\ux = \ux_1 M\ux_2$, and that the corresponding
triple $(m,m_1,m_2)$ is a node of the Markoff tree. Using
Cayley-Hamilton's theorem, we find
\begin{equation}
 \label{M:propStree:eq4}
 \begin{aligned}
 \ux_1 M \ux
   &= (\ux_1M)^2\ux_2
   = \big(\trace(\ux_1M)\ux_1M-\det(\ux_1M)I\big)\ux_2
   = 3m_1\ux - \ux_2,\\
 \ux M \ux_2
  &= \ux_1(M\ux_2)^2
   = \ux_1 \big(\trace(M\ux_2)M\ux_2-\det(M\ux_2)I\big)
   = 3m_2\ux - \ux_1.
 \end{aligned}
\end{equation}
Since $\ux_1$, $\ux_2$ and $\ux$ are symmetric matrices in
$\SL_2(\bZ)$ and since $M\in\SL_2(\bZ)$, we conclude that these
products are also symmetric matrices in $\SL_2(\bZ)$.  Moreover, if
we write
\[
 \ux_1 M \ux = \matrice{m'_2}{k'_2}{k'_2}{\ell'_2}
 \et
 \ux M \ux_2 = \matrice{m'_1}{k'_1}{k'_1}{\ell'_1},
\]
then we obtain $m'_2=3m_1m-m_2$ and $m'_1=3m_2m-m_1$ showing that
the triples $(m'_2,m_1,m)$ and $(m'_2,m,m_2)$ associated
respectively to the left and right successors of $(\ux,\ux_1,\ux_2)$
are respectively the left and right successors of $(m,m_1,m_2)$ in
the Markoff tree.  By recurrence, this proves all the assertions of
the proposition besides the constrains on the coefficients of the
matrices. To prove the latter, suppose that the node
$(\ux,\ux_1,\ux_2)$ satisfies conditions of the form
\begin{equation}
 \label{M:propStree:eq5}
 \varphi(\ux) \ge \varphi(\ux_i) \ge c \quad \text{for $i=1,2$,}
\end{equation}
for some constant $c\ge 0$ and some linear form $\varphi$ on
the space of $2\times 2$ matrices.  Then, using the fact that
$m_1$ and $m_2$ are positive (because $(m,m_1,m_2) \in
\Sigma^*$), the relations \eqref{M:propStree:eq4} lead to
\[
 \begin{aligned}
 \varphi(\ux_1 M \ux)
   &= 3m_1\varphi(\ux) - \varphi(\ux_2)
   \ge \varphi(\ux) \ge \varphi(\ux_1) \ge c,\\
 \varphi(\ux M \ux_2)
   &= 3m_2\varphi(\ux) - \varphi(\ux_1)
   \ge \varphi(\ux) \ge \varphi(\ux_2) \ge c,
 \end{aligned}
\]
showing by induction on the level that \eqref{M:propStree:eq5} holds
for each node of the tree \eqref{M:propStree:tree} as soon as it
holds for its root. Since the latter satisfies $m\ge m_i\ge 1$,
$k\ge k_i\ge 1$ and $\ell\ge \ell_i\ge 1$ for $i=1,2$, we conclude
that each node of the tree meets these conditions and so consists of
matrices with positive entries. Moreover, since the root also
satisfies $m-k\ge m_i-k_i \ge 0$ and $2k-m\ge 2k_i-m_i\ge 0$ for
$i=1,2$, each node meets these additional conditions and in
particular satisfies $k\le m\le 2k$.  Finally, since $(m, m_1, m_2)
\in \Sigma^*$, we have $m> \max\{m_1,m_2\}\ge 1$, thus $m\ge 2$ and,
from $1=\det(\ux)=m\ell-k^2$, we deduce that $\ell = (k^2+1)/m \le
m+1/m < m+1$ and therefore $\ell\le m$.
\end{proof}

For each node $\um = (m,m_1,m_2)$ of \eqref{M:Mtree}, we denote by
\begin{equation}
 \label{M:ux_um}
 \ux_\um = \matrice{m}{k}{k}{\ell}
\end{equation}
the first component of the corresponding node
\eqref{M:propStree:eq2} of the tree \eqref{M:propStree:tree}, and we
extend this definition to all of $\Sigma$ by putting
\begin{equation}
 \label{M:ux_um_init}
 \ux_{(1,1,1)} = \matrice{1}{1}{1}{2}
 \et
 \ux_{(2,1,1)} = \matrice{2}{1}{1}{1}.
\end{equation}
Then, for each $\um\in\Sigma$, we define
\begin{equation}
 \label{M:F_um}
 F_\um(U,T)
 = \begin{pmatrix}T &-U\end{pmatrix}\ux_\um M
      \begin{pmatrix}U \\ T\end{pmatrix}
 = mT^2 + (3m-2k)TU + (\ell-3k) U^2,
\end{equation}
using the notation \eqref{M:ux_um}.  Since $\det(\ux_\um) =
m\ell-k^2 = 1$, we find that $\disc(F_\um) = 9m^2-4$.  Since
$\disc(F_\um)\equiv 2 \mod 3$, the form $F_\um$ is irreducible over
$\bQ$.  Therefore it factors as a product
\[
 F_\um(U,T) = m (T-\alpha_\um U) (T-\alphabar_\um U)
\]
where
\begin{equation}
 \label{M:alpha_m}
 \alpha_\um = \frac{2k-3m + \sqrt{9m^2-4}}{2m}
 \et
 \alphabar_\um = \frac{2k-3m - \sqrt{9m^2-4}}{2m}
\end{equation}
are conjugate quadratic real numbers.

In his presentation of Markoff's theory, Cassels also defines
quadratic forms indexed by solutions $\um$ of Markoff's equation,
except that, assuming the uniqueness conjecture, he denotes them
simply $F_m$ where $m$ is the largest entry of $\um$, the conjecture
being that this entry determines uniquely the solution (see
\cite[p.~33]{Ca} or \cite[Appendix B]{Bo}).  In view of the
discussion in \cite[Ch.~II, \S4]{Ca}, the corollary below shows that
the above forms $F_\um$ are equivalent to the corresponding forms
defined by Cassels.

\begin{corollary}
 \label{M:cor}
For each $\um = (m,m_1,m_2) \in \Sigma$, the off-diagonal entry $k$
of $\ux_\um$ satisfies
\begin{equation}
 \label{M:cor:eq}
 k \equiv \frac{m_1}{m_2} \equiv \frac{-m_2}{m_1} \mod m
 \et
 0< k\le m.
\end{equation}
\end{corollary}

Note that the condition \eqref{M:cor:eq} makes sense since each
triple of $\Sigma$ has pairwise relatively prime components
\cite[Ch.~II, \S3, Lemma 5]{Ca}.  It also determines $k$ uniquely.

\begin{proof}
This is readily checked when $\um$ is $(1,1,1)$ or $(2,1,1)$.  Now,
assume that $\um = (m,m_1,m_2)$ is non-degenerate and write the
corresponding triple of symmetric matrices $(\ux,\ux_1,\ux_2)$ in
the form \eqref{M:propStree:eq2}.  Since $\ux=\ux_\um$, this
notation is consistent with \eqref{M:ux_um}. Then, by Proposition
\ref{M:propStree}, we have $0< k\le m$. Since $\ux$, $\ux_1$ and
$\ux_2$ are symmetric, taking the transpose of both sides of the
equality $\ux = \ux_1 M \ux_2$ gives $\ux = \ux_2 \tM \ux_1$, and so
we obtain $\ux\ux_2^{-1}=\ux_1M$ and $\ux\ux_1^{-1}=\ux_2\tM$.
Comparing the upper right entries in the latter matrix equalities,
we find that $km_2-mk_2=m_1$ and $km_1-mk_1 = -m_2$ from which the
requested congruences follow.
\end{proof}

Combining Theorems II and III in Chapter II of \cite{Ca}, we then
recover the following main results of Markoff \cite{Ma,Mb}.

\begin{theorem}[Markoff, 1879-80]
 \label{M:thmM}
The real numbers $\alpha_\um$ with $\um\in\Sigma$ form a system of
representatives of the equivalence classes of real numbers $\xi$
with $\nu(\xi)> 1/3$, while the forms $F_\um$ with $\um \in \Sigma$
constitute a system of representatives of the equivalence classes of
real indefinite binary quadratic forms $F$ with $\mu(F) /
\sqrt{\disc(F)} > 1/3$.  Moreover, for each $\um = (m,m_1,m_2) \in
\Sigma$, the numbers $\alpha_\um$ and $\alphabar_\um$ are equivalent
and we have
\[
\nu(\alpha_\um) = \nu(\alphabar_\um) =
\frac{\mu(F_\um)}{\sqrt{\disc(F_\um)}} = \frac{1}{\sqrt{9-4m^{-2}}}.
\]
\end{theorem}

%
%

\section{Extremal numbers}
\label{sec:ext}

Let $\cP$ denote the set of $2\times 2$ matrices with relatively
prime integer coefficients.  It is a group for the product $*$ given
by $\uy_1*\uy_2 = c^{-1}\uy_1\uy_2$ where $c$ is the greatest
positive common divisor of the coefficients of $\uy_1\uy_2$. This
group contains $\GL_2(\bZ)$ as a subgroup, and its quotient
$\cP/\{\pm I\}$ is isomorphic to $\mathrm{PGL}_2(\bQ)$.  With this
notation, we state the following characterization of extremal
numbers reproduced from \cite[Lemma 3.1]{Rcfrac}, which collects
results from \cite{RcubicI, Rdio}.

\begin{proposition}
 \label{ext:propExt}
Let $\xi$ be an extremal real number.  Then, there exists an
unbounded sequence of symmetric matrices $(\ux_i)_{i\ge 1}$ in $\cP$
such that, for each $i\ge 1$, we have
\begin{equation}
 \label{ext:propExt:eq1}
 \|\ux_{i+1}\| \asymp \|\ux_i\|^\gamma, \quad
 \|(\xi,-1)\ux_i\| \asymp \|\ux_i\|^{-1}
 \et
 |\det \ux_i| \asymp 1,
\end{equation}
with implied constants that are independent of $i$. Such a sequence
$(\ux_i)_{i\ge 1}$ is uniquely determined by $\xi$ up to its first
terms and up to multiplication of each of its terms by $\pm 1$.
Moreover, for any such sequence, there exists a non-symmetric and
non-skew-symmetric matrix $M\in\cP$ such that
\begin{equation}
 \label{ext:propExt:eq2}
 \ux_{i+2}
 = \pm
   \begin{cases}
     \ux_{i+1}*M*\ux_i &\text{if $i$ is odd,}\\
     \ux_{i+1}*\tM*\ux_i &\text{if $i$ is even,}
   \end{cases}
\end{equation}
for any sufficiently large index $i$.  Conversely, if $(\ux_i)_{i\ge
1}$ is an unbounded sequence of symmetric matrices in $\cP$ which
satisfies a recurrence relation of the type \eqref{ext:propExt:eq2}
for some non-symmetric matrix $M\in\cP$, and if
\begin{equation}
 \label{ext:propExt:eq3}
 \|\ux_{i+2}\| \gg \|\ux_{i+1}\|\, \|\ux_i\|
 \et
 |\det \ux_i| \ll 1,
\end{equation}
then $(\ux_i)_{i\ge 1}$ also satisfies the estimates
\eqref{ext:propExt:eq1} for some extremal real number $\xi$.
\end{proposition}

In the above statement, the choice of a norm for matrices is
secondary since it only affects the implied constants in all
estimates. However, for definiteness, we choose the norm $\|\ux\|$
of a matrix $\ux$ with real coefficients to be the largest absolute
value of its coefficients.  Then, for an extremal number $\xi$ with
a corresponding unbounded sequence of symmetric matrices
$(\ux_i)_{i\ge 1}$ in $\cP$ satisfying \eqref{ext:propExt:eq1}, we
find
\[
 \|(\xi,-1)\ux_i\|
 =
 \max\{ |x_{i,0}\xi-x_{i,1}|,\, |x_{i,1}\xi-x_{i,2}| \}
 \quad
 \textrm{upon writing}
 \quad
 \ux_i=\matrice{x_{i,0}}{x_{i,1}}{x_{i,1}}{x_{i,2}},
\]
and therefore $\xi = \lim_{i\to\infty} x_{i,1}/x_{i,0} =
\lim_{i\to\infty} x_{i,2}/x_{i,1}$.\\

It can be shown directly from the definition that the set of
extremal numbers is stable under the action of $\GL_2(\bQ)$ by
linear fractional transformations on $\bR\setminus\bQ$
\cite[\S2]{Rcfrac}. In particular, it is stable under the action of
the subgroup $\GL_2(\bZ)$.  The next corollary shows how the latter
action affects the corresponding sequences of symmetric matrices
$(\ux_i)_{i\ge 1}$ and the corresponding matrices $M$.

\begin{corollary}
 \label{ext:cor}
Let $\xi$ be an extremal number, let $(\ux_i)_{i\ge 1}$ be an
unbounded sequence of symmetric matrices in $\cP$ satisfying
\eqref{ext:propExt:eq1} and let $M\in\cP$ such that
\eqref{ext:propExt:eq2} holds.  For any $g = \matrice{a}{b}{c}{d}
\in \SL_2(\bZ)$, the number $\xi':=g\cdot\xi$ is also extremal with
corresponding sequence $(\ux'_i)_{i\ge 1}$ and matrix $M'$ given by
\begin{equation}
 \label{ext:cor:eq1}
 \ux_i' = \trans(g')^{-1}\ux_i(g')^{-1}
 \et
 M' = g' M \trans g',
 \quad\text{where}\quad g'=\matrice{a}{-b}{-c}{d}.
\end{equation}
\end{corollary}

\begin{proof}
It is clear that the above matrices $\ux'_i$ and $M'$ belong to
$\cP$ and satisfy the recurrence relation \eqref{ext:propExt:eq2}
instead of $\ux_i$ and $M$. Moreover, the matrices $\ux'_i$ are
symmetric while $M'$ is both non-symmetric and non-skew-symmetric.
We also find that $\|\ux'_i\|\asymp \|\ux_i\|$,
\[
 \|(\xi',-1)\ux'_i\|
 = |c\xi+d|^{-1} \|(\xi,-1)\ux_i(g')^{-1}\|
 \asymp \|(\xi,-1)\ux_i\|,
\]
and $\det(\ux'_i)=\det(\ux_i)$.  Therefore $(\ux'_i)_{i\ge 1}$ and
$\xi'$ also satisfy \eqref{ext:propExt:eq1} instead of
$(\ux_i)_{i\ge 1}$ and $\xi$.  In particular, $(\ux'_i)_{i\ge 1}$
satisfies \eqref{ext:propExt:eq3} and so, by the last part of
Proposition \ref{ext:propExt}, it obeys \eqref{ext:propExt:eq1} for
some extremal number $\xi''$ instead of $\xi$.  This forces
$\xi'=\xi''$, and so $\xi'$ is extremal.
\end{proof}

It follows from Proposition \ref{ext:propExt} that the matrix
$M\in\cP$ attached to an extremal number $\xi$ is uniquely
determined by $\xi$ within the set $\{M,-M,\tM,-\tM\}$.  When the
sequence of symmetric matrices attached to $\xi$ is contained in
$\SL_2(\bZ)$, the matrix $M$ also belongs to $\SL_2(\bZ)$ and the
recurrence relation \eqref{ext:propExt:eq2} can be put in simpler
form. Then, applying an identity of Fricke  like Cohn in \cite{Co}, we
obtain:

\begin{lemma}
 \label{ext:lemma:Fricke}
Let $\xi$ be an extremal number with a corresponding sequence of
symmetric matrices $(\ux_i)_{i\ge 1}$ in $\SL_2(\bZ)$.  Choose
$M\in\SL_2(\bZ)$ and the above sequence so that, for each $i\ge 1$,
we have
\[
 \ux_{i+2} = \ux_{i+1}M_{i+1}\ux_i
 \where
 M_i=
  \begin{cases} M &\text{if $i$ is even,}\\
              \tM &\text{if $i$  is odd.}
  \end{cases}
\]
Then, for each $i\ge 1$, the traces $q_i := \trace(\ux_iM_i) \in \bZ$ satisfy
\begin{equation}
 \label{ext:prop:Fricke:eq1}
 q_{i+2}^2+q_{i+1}^2+q_{i}^2 = q_{i+2}q_{i+1}q_{i} + \trace(\tM\,M^{-1})+2.
\end{equation}
\end{lemma}

\begin{proof}
In \cite{Fr}, Fricke shows that for any $A, B \in \SL_2(\bR)$ we
have
\[
 \trace(A)^2 + \trace(B)^2 + \trace(AB)^2
 = \trace(A)\trace(B)\trace(AB) +\trace(ABA^{-1}B^{-1}) + 2.
\]
Putting $A=\ux_{i+1}M_{i+1}$ and $B=\ux_iM_i$, the recurrence
relation gives $AB = \ux_{i+2}M_i = \ux_{i+2}M_{i+2}$ and so
$\trace(AB)=q_{i+2}$.  Since $\ux_{i+2}$ is symmetric, we also find
$AB = \tux_{i+2}M_i = \ux_i M_i \ux_{i+1}M_i = BA M_{i+1}^{-1}M_i$
and so $\trace(ABA^{-1}B^{-1}) = \trace(M_{i+1}^{-1}M_i)$.  The
conclusion follows since $\trace(M_{i+1}^{-1}M_i) =
\trace(\tM_{i+2}^{-1}\tM_{i+1}) = \trace(M_{i+2}^{-1}M_{i+1})$ is
independent of $i$.
\end{proof}

We observed in \cite{RcubicII} that the arithmetic of extremal
numbers is particularly simple when the corresponding sequence of
symmetric matrices $(\ux_i)_{i\ge 1}$ is contained in $\GL_2(\bZ)$
and the lower right entry of the corresponding matrix $M$ is $0$.
When all these matrices belong to $\SL_2(\bZ)$, the preceding result
applies and we find:

\begin{lemma}
 \label{ext:lemma:Eu}
Let $u$ be a non-zero integer and let $\cE_u^+$ denote the set of
all extremal numbers with a corresponding sequence of symmetric
matrices $\ux_i = \matrice{x_{i,0}}{x_{i,1}}{x_{i,1}}{x_{i,2}}$ in
$\SL_2(\bZ)$ satisfying, for each $i\ge 1$,
\begin{equation}
 \label{ext:lemma:Eu:eq0}
 \ux_{i+2} = \ux_{i+1}M_{i+1}\ux_i
 \where
 M_i=\matrice{u}{(-1)^{i}}{(-1)^{i+1}}{0}.
\end{equation}
Then, the set $\cE_u^+ = \cE_{-u}^+$ is empty if $u\neq \pm 3$.
Moreover, if $\xi\in\cE_3^+$, then, upon choosing the matrices
$\ux_i$ as above, each triple $(x_{i+2,0},x_{i+1,0},x_{i,0})$ is a
solution of Markoff's equation \eqref{Markoffeq}.
\end{lemma}

\begin{proof}
Let $\xi\in\cE_u^+$.  Using the notation of the lemma, a simple
computation shows that the matrix $M := M_2$ satisfies
$\trace(\tM\,M^{-1})=-2$ and that, for each $i\ge 1$, we have
$\trace(\ux_i M_i)=ux_{i,0}$. Therefore, Lemma
\ref{ext:lemma:Fricke} gives
\begin{equation}
 \label{ext:lemma:Eu:eq1}
 x_{i+2,0}^2+x_{i+1,0}^2+x_{i,0}^2 = u\,x_{i+2,0}\,x_{i+1,0}\,x_{i,0}
\end{equation}
for each $i\ge 1$.  Since $-1$ is not a square modulo $3$ and since
$1=\det(\ux_i)\equiv -x_{i,1}^2 \mod x_{i,0}$, we also note that
$x_{i,0}$ is prime to $3$ for each $i\ge 1$.  Then, looking at the
equation \eqref{ext:lemma:Eu:eq1} modulo $3$, we deduce that $u$ is
divisible by $3$ and so, each triple $(u/3)(x_{i+2,0}, x_{i+1,0},
x_{i,0})$ provides a solution of Markoff's equation in integers not
all zero.  Since each such solution has relatively prime entries,
this is possible only if $u=\pm 3$.
\end{proof}

\begin{lemma}
 \label{ext:lemma:equiv}
Two elements $\xi$ and $\xi'$ of\/ $\cE_3^+$ are equivalent (under
$\GL_2(\bZ)$\/) if and only if $\xi'=\pm\, \xi + b$ for some
$b\in\bZ$. Each element of\/ $\cE_3^+$ is equivalent to one and only
one element of\/ $\cE_3^+$ in the open interval $(1/2,1)$.
\end{lemma}

\begin{proof}
The second assertion follows from the first since, for each
$\xi\in\bR\setminus\bQ$, there is a unique integer $b$ and a unique
choice of sign such that $\pm\,\xi+b \in (1/2,1)$.  To prove the
first assertion, suppose that $\xi \in \cE_3^+$ and let $g =
\matrice{a}{b}{c}{d} \in \GL_2(\bZ)$.  By Corollary \ref{ext:cor},
we have $g\cdot\xi \in \cE_3^+$ if and only if
\[
 \matrice{a}{-b}{-c}{d} \matrice{3}{1}{-1}{0} \matrice{a}{-c}{-b}{d}
 = \epsilon_1 \matrice{3}{\epsilon_2}{-\epsilon_2}{0},
\]
for some choices of $\epsilon_1,\epsilon_2\in\{1,-1\}$. Equating
coefficients, this translates into the conditions $3a^2 =
3\epsilon_1$, $3c^2=0$ and $\det(g)\pm\,3ac = \epsilon_1\epsilon_2$
which mean $a=\epsilon_1=1$, $c=0$, $d=\epsilon_2$ and impose no
restriction on $b$.  For such $a$, $c$ and $d$, we find
$g\cdot\xi=\epsilon_2(\xi+b)$.
\end{proof}

A \emph{zigzag} in the tree \eqref{M:Mtreex} is a sequence of nodes
$\um^{(1)}, \um^{(2)}, \um^{(3)}, \dots$ of that tree such that, for
each $i\ge 1$, the node $\um^{(i+1)}$ is a successor of $\um^{(i)}$
on some side (left or right) and $\um^{(i+2)}$ is a successor of
$\um^{(i+1)}$ on the other side.  A \emph{maximal zigzag} is a
zigzag $\um^{(1)}, \um^{(2)}, \um^{(3)}, \dots$ which cannot be
extended by inserting an ancestor of $\um^{(1)}$ as the first
element.  With the convention that the root $(2,1,1)$ has no
ancestor in \eqref{M:Mtreex}, it follows that each $\um\in\Sigma^*$
is the first element of a unique maximal zigzag. Examples of maximal
zigzags in \eqref{M:Mtreex} are
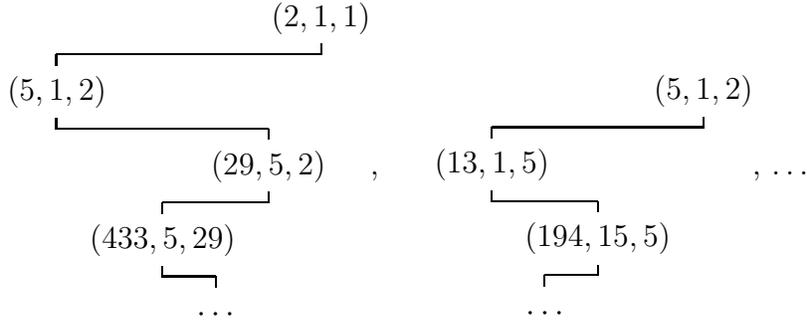
\[
  \begin{matrix}
  \hspace{100pt} (2,1,1)\\[-3pt]
  \lacc\\[3pt]
  (5,1,2)\hspace{100pt}\\[-3pt]
  \racci\hspace{20pt}\\[3pt]
  \hspace{60pt}(29,5,2)\\[-3pt]
  \hspace{20pt}\laccii \\[3pt]
  (433,5,29) \hspace{20pt}\\[-3pt]
  \racciii\\[3pt]
  \hspace{20pt}\cdots
  \end{matrix}
  ,\,
  \begin{matrix}
  \\[13pt]
  \hspace{100pt}(5,1,2)\\[-3pt]
  \hspace{20pt}\lacci\\[3pt]
  (13,1,5)\hspace{60pt}\\[-3pt]
  \raccii\hspace{20pt} \\[3pt]
  \hspace{20pt}(194,15,5) \\[-3pt]
  \lacciii\\[3pt]
  \cdots \hspace{20pt}
  \end{matrix}
  ,\,\dots
\]
Recall that, in Section \ref{sec:M}, we attached a symmetric matrix
$\ux_\um \in \SL_2(\bZ)$ to each $\um\in\Sigma$.  Thus, each maximal
zigzag in \eqref{M:Mtreex} leads to a sequence of symmetric matrices
in $\SL_2(\bZ)$. We can now state and prove the main result of this
section.

\begin{theorem}
 \label{ext:thm_xi}
Given $\um\in\Sigma^*$, consider the maximal zigzag\/
$\um=\um^{(1)}, \um^{(2)}, \um^{(3)}, \dots$ in the tree
\eqref{M:Mtreex} originating from $\um$. Then
$(\ux_{\um^{(i)}})_{i\ge 1}$ is a sequence of symmetric matrices in
$\SL_2(\bZ)$ corresponding to an extremal number $\xi_\um$ in
$\cE_3^+ \cap(1/2,1)$ and we have
\begin{equation}
 \label{ext:thm_xi:lim}
 \xi_\um
 = \lim_{i\to\infty} \alpha_{\um^{(i)}}
 = \lim_{i\to\infty} (\alphabar_{\um^{(i)}}+3)
\end{equation}
in terms of the quadratic numbers given by \eqref{M:alpha_m}. Each
element of $\cE_3^+$ is equivalent to $\xi_\um$ for one and only one
$\um\in\Sigma^*$.
\end{theorem}

\begin{proof}
Let $M=\matrice{3}{1}{-1}{0}$ be as in Proposition \ref{M:propStree}
and let $(\um^{(i)})_{i\ge 1}$ be a maximal zigzag in
\eqref{M:Mtreex} originating from a point $\um=\um^{(1)}$ in
$\Sigma^*$. For simplicity, we simply write $\ux_i$ to denote the
matrix $\ux_{\um^{(i)}}$. If, for some index $i$, the point
$\um^{(i+1)}$ is the left successor of $\um^{(i)}$, then the node of
the tree \eqref{M:propStree:tree} corresponding to $\um^{(i+1)}$
takes the form $(\ux_{i+1},*,\ux_i)$ and, as $\um^{(i+2)}$ is the
right successor of $\um^{(i+1)}$, we find that $\ux_{i+2} =
\ux_{i+1}M\ux_i$. Similarly, if $\um^{(i+1)}$ is the right successor
of $\um^{(i)}$, then the node of \eqref{M:propStree:tree}
corresponding to $\um^{(i+1)}$ takes the form $(\ux_{i+1},\ux_i,*)$
and $\um^{(i+2)}$ is the left successor of $\um^{(i+1)}$, thus
$\ux_{i+2} = \ux_iM\ux_{i+1} = \ux_{i+1}\tM\ux_i$.  As, the parity
of $i$ decides which alternative holds, we deduce that the condition
\eqref{ext:propExt:eq2} of Proposition \ref{ext:propExt} is
satisfied for each $i\ge 1$ with the present choice of $M$ or with
$M$ replaced by its transpose $\tM$.  The above considerations also
show that, for each $i\ge 1$, the node of \eqref{M:propStree:tree}
corresponding to $\um^{(i+2)}$ is either $(\ux_{i+2}, \ux_{i+1},
\ux_i)$ or $(\ux_{i+2}, \ux_i, \ux_{i+1})$ and so $\um^{(i+2)}$ can
be described as the node of the Markoff tree \eqref{M:Mtreex} formed
by the upper left entries of $\ux_{i+2}$, $\ux_{i+1}$ and $\ux_{i}$.

To verify the conditions \eqref{ext:propExt:eq3} of Proposition
\ref{ext:propExt}, we write $\ux_i =
\matrice{m_i}{k_i}{k_i}{\ell_i}$. With this notation, Proposition
\ref{M:propStree} gives $\|\ux_i\| = m_i$ and $k_i\le m_i \le 2k_i$
for each $i\ge 1$.  Thus, if $\ux_{i+2} = \ux_{i+1}M\ux_i$, we find
that
\[
 m_{i+2}
  = (3m_{i+1}-k_{i+1})m_i+m_{i+1}k_i
  \ge (5/2)m_{i+1}m_i.
\]
Otherwise, we have $\ux_{i+2} = \ux_iM\ux_{i+1}$ and the same
computation applies with the indices $i$ and $i+1$ permuted. This
means that $\|\ux_{i+2}\| \ge (5/2)\, \|\ux_{i+1}\|\, \|\ux_i\|$ for
each $i\ge 1$.  Since $\det(\ux_i)=1$ for each $i$, the conditions
\eqref{ext:propExt:eq3} of Proposition \ref{ext:propExt} are
fulfilled and therefore $(\ux_i)_{i\ge 1}$ satisfies the conditions
\eqref{ext:propExt:eq1} of the same proposition for some extremal
number $\xi=\xi_\um$.  We have $\xi_\um\in \cE_3^+$ by definition,
and moreover $\xi_\um = \lim_{i\to\infty} k_i/m_i \in [1/2,1]$. Then
\eqref{ext:thm_xi:lim} follow from the formulas \eqref{M:alpha_m}
and, as $\xi_\um$ is irrational, we conclude that $\xi_\um \in
\cE_3^+ \cap(1/2,1)$. The first assertion of the theorem is proved.

Now assume that $\xi_\um=\xi_\un$ for some $\un\in\Sigma^*$, and let
$(\un^{(i)})_{i\ge 1}$ denote the maximal zigzag starting with
$\un^{(1)}=\un$. Then, $(\ux_{\um^{(i)}})_{i\ge 1}$ and
$(\ux_{\un^{(i)}})_{i\ge 1}$ are two sequences of symmetric matrices
with positive entries corresponding to the same extremal number. By
Proposition \ref{ext:propExt}, this is possible if and only if there
exists an integer $s$ such that $\ux_{\um^{(i)}} =
\ux_{\un^{(i+s)}}$ for each sufficiently large $i$. However, we
observed that, for each $i\ge 1$, the triple $\um^{(i+2)}$ is the
node of \eqref{M:Mtreex} formed by the upper left entries of
$\ux_{\um^{(i+2)}}$, $\ux_{\um^{(i+1)}}$ and $\ux_{\um^{(i)}}$.
Similarly, $\un^{(i+2)}$ is formed by the upper left entries of
$\ux_{\un^{(i+2)}}$, $\ux_{\un^{(i+1)}}$ and $\ux_{\un^{(i)}}$. This
forces $\um^{(i)} = \un^{(i+s)}$ for each sufficiently large $i$ and
therefore $\um=\un$ because each zigzag in \eqref{M:Mtreex} is
contained in a unique maximal zigzag.

Lemma \ref{ext:lemma:equiv} together with the preceding observation
reduce the last assertion of the theorem to proving that each
element of $\cE_3^+ \cap(1/2,1)$ is equal to $\xi_\um$ for some
$\um\in\Sigma^*$.  To this end, we fix a point $\xi \in \cE_3^+
\cap(1/2,1)$ and a corresponding sequence $(\ux_i)_{i\ge 1}$ of
symmetric matrices in $\SL_2(\bZ)$ obeying the recurrence relation
\eqref{ext:lemma:Eu:eq0} of Lemma \ref{ext:lemma:Eu} with $u=3$.
Using the notation of that lemma for the entries of $\ux_i$, we have
$\xi = \lim_{i\to\infty} x_{i,1}/x_{i,0}$. Since $\xi$ belongs to
$(1/2,1)$, the ratio $x_{i,1}/x_{i,0}$ must also belong to that
interval for each sufficiently large integer $i$.  Without loss of
generality, we may assume that this already holds for each $i\ge 1$.
Upon multiplying $\ux_1$ and $\ux_2$ by $\pm1$ and adjusting the
following $\ux_i$ so that \eqref{ext:lemma:Eu:eq0} continues to
hold, we may also assume that $x_{1,0}$ and $x_{2,0}$ are positive.
Then a simple recurrence argument based on \eqref{ext:lemma:Eu:eq0}
shows that $x_{i,0}> \max\{x_{i-1,0},\,x_{i-2,0}\}> 0$ for each
$i\ge 3$. By Lemma \ref{ext:lemma:Eu}, this means that, for each
$i\ge 3$, exactly one of the points $(x_{i,0}, x_{i-1,0},
x_{i-2,0})$ or $(x_{i,0}, x_{i-2,0}, x_{i-1,0})$ is a node
$\um^{(i)}$ of the tree \eqref{M:Mtreex}.  In particular, the
integers $x_{i,0}$, $x_{i-1,0}$, $x_{i-2,0}$ are pairwise relatively
prime.

We claim that $\ux_i = \ux_{\um^{(i)}}$ for each $i\ge 3$. Since the
symmetric matrices $\ux_i$ and $\ux_{\um^{(i)}}$ have the same upper
left entries and the same determinant, this reduces to showing that
the off-diagonal entry $k$ of $\ux_{\um^{(i)}}$ is $x_{i,1}$. In the
notation of Lemma \ref{ext:lemma:Eu} (with $u=3$), we have
$\ux_i\ux_{i-2}^{-1} = \ux_{i-1}M_{i-1}$ which, by comparing the
upper right entries of the matrices on both sides (as in the proof
of Corollary \ref{M:cor}), gives $x_{i,1}x_{i-2,0}-x_{i,0}x_{i-2,1}
= (-1)^{i-1}x_{i-1,0}$ and therefore
\begin{equation}
 \label{ext:thm_xi:eq}
 x_{i,1} \equiv (-1)^{i-1} \frac{x_{i-1,0}}{x_{i-2,0}} \mod x_{i,0}.
\end{equation}
By comparison with the conditions that Corollary \ref{M:cor} imposes
on $k$, this leads to $k\equiv \pm\,x_{i,1} \mod x_{i,0}$.  As
Proposition \ref{M:propStree} gives $x_{i,0}/2\le k\le x_{i,0}$ and
as we know that $x_{i,0}/2< x_{i,1}< x_{i,0}$, we conclude that
$k=x_{i,1}$ and the claim is proved.

Comparing the congruence \eqref{ext:thm_xi:eq} with those of
\eqref{M:cor:eq} shows moreover that, for $i\ge 3$, we have
$\um^{(i)} = (x_{i,0}, x_{i-1,0}, x_{i-2,0})$ if $i$ is odd and
$\um^{(i)} = (x_{i,0}, x_{i-2,0}, x_{i-1,0})$ if $i$ is even. Since
$\um^{(i+1)}$ has two coordinates in common with $\um^{(i)}$ and a
larger first coordinate, this implies that, in the Markoff tree
\eqref{M:Mtreex}, $\um^{(i+1)}$ is the left successor of $\um^{(i)}$
if $i$ is odd, and its right successor if $i$ is even (see
\cite[Ch.~II, \S3]{Ca}). Thus, the sequence $(\um^{(i)})_{i\ge 3}$
is a zigzag in \eqref{M:Mtreex} and $(\ux_{\um^{(i)}})_{i\ge 3}$ is
a sequence of symmetric matrices associated to the extremal number
$\xi$.  We conclude that $\xi=\xi_\um$ where $\um$ is the first
element of the maximal zigzag containing $(\um^{(i)})_{i\ge 3}$.
\end{proof}

The main goal of this paper is to show that the set $\{\xi_\um\,;\,
\um\in\Sigma^*\}$ constitutes a system of representatives of the
equivalence classes of extremal numbers $\xi$ with $\nu(\xi)=1/3$.
By Lemma \ref{ext:lemma:equiv}, we know that they belong to distinct
equivalence classes.  The next step is to show that
$\nu(\xi_\um)=1/3$ for each $\um\in\Sigma^*$. This will be achieved
in \S\ref{sec:reduction}.

%
%

\section{Conjugates of an extremal number}
\label{sec:conj}

This section deals with approximation to extremal numbers by
quadratic real numbers, and introduces the notion of conjugates of
an extremal number, a concept which will play an important role in
the sequel.  With respect to notation, we define the \emph{norm}
$\|F\|$ of a polynomial $F$ over $\bR$ to be the largest absolute
value of its coefficients, and we define the \emph{height}
$H(\alpha)$ of an algebraic number $\alpha$ to be the norm of its
minimal polynomial in $\bZ[T]$.

Throughout the section, we fix an arbitrary extremal number $\xi$, a
corresponding unbounded sequence of symmetric matrices
$(\ux_i)_{i\ge 1}$ in $\cP$ satisfying the condition
\eqref{ext:propExt:eq1} of Proposition \ref{ext:propExt}, and a
matrix $M \in \cP$ which is assumed to satisfy
\eqref{ext:propExt:eq2} for each $i\ge 1$ (this condition on the
range of $i$ carries no loss of generality). For each $i\ge 1$, we
write
\[
 J = \matrice{0}{1}{-1}{0},
 \quad
 M = \matrice{a}{b}{c}{d},
 \quad
 \ux_i = \matrice{x_{i,0}}{x_{i,1}}{x_{i,1}}{x_{i,2}}
 \et
 X_i = \|\ux_i\|.
\]
We also define new matrices
\[
 W_i = \ux_i * M_i
 \where
 M_i =
  \begin{cases} M &\text{if $i$ is even,}\\
              \tM &\text{if $i$ is  odd,}
  \end{cases}
\]
and real quadratic forms
\[
 F_i(U,T)
 = -\begin{pmatrix}U &T\end{pmatrix} J W_i
    \begin{pmatrix}U\\ T\end{pmatrix}
 \et
 G_i(U,T)
 = -\begin{pmatrix}U &T\end{pmatrix} J
    \begin{pmatrix}1 &\xi\\ \xi &\xi^2 \end{pmatrix}
    M_i
    \begin{pmatrix}U\\ T\end{pmatrix}
\]

It is clear from the above definition that $G_i$ depends only on the
parity of $i$.  A short computation gives the following formulas.

\begin{lemma}
 \label{conj:lemmaG}
For each integer $i\ge 1$, we have
\begin{equation}
 \label{conj:lemmaG:G'G''}
 G_i(U,T)
  =
  \begin{cases}
  G'(U,T) := (c+d\xi)(T-\xi U)(T-\xi'U)
   &\text{if $i$ is odd,}\\[3pt]
  G''(U,T) := (b+d\xi)(T-\xi U)(T-\xi''U)
   &\text{if $i$ is even,}
  \end{cases}
\end{equation}
where
\begin{equation}
 \label{conj:lemmaG:xi'xi''}
 \xi' = -\frac{a+b\xi}{c+d\xi}
 \et
 \xi'' = -\frac{a+c\xi}{b+d\xi}\, \cdot
\end{equation}
The sets $\{\xi',\xi''\}$ and $\{\pm G',\pm G''\}$ depend only on
$\xi$.  Moreover, $\xi$, $\xi'$ and\/ $\xi''$ are three distinct
extremal numbers.
\end{lemma}

\begin{proof}
The second assertion of the lemma follows from the facts that $M$ is
uniquely determined by $\xi$ within the set $\{\pm M, \pm \tM\}$
(see \S\ref{sec:ext}), and that replacing $M$ by $\pm M$ or by
$\pm\tM$ just permutes the elements of $\{\xi',\xi''\}$ and $\{\pm
G',\pm G''\}$. The real numbers $\xi'$ and $\xi''$ are extremal
because they belong to the $\GL_2(\bQ)$-orbit of $\xi$ (see
\cite[\S2]{Rcfrac}).  Finally, the numbers $\xi$, $\xi'$ and $\xi''$
are distinct because $\xi$ is not quadratic over $\bQ$ and, by
Proposition \ref{ext:propExt}, $M$ is neither symmetric nor
skew-symmetric.
\end{proof}

\begin{definition}
 \label{conj:def}
The extremal numbers $\xi'$ and $\xi''$ given by
\eqref{conj:lemmaG:xi'xi''} are called the \emph{conjugates} of
$\xi$ while the polynomials $G'$ and $G''$ given by
\eqref{conj:lemmaG:G'G''} are called the \emph{real quadratic forms
associated} to $\xi$.
\end{definition}

For example, the extremal numbers $\xi_\um$ constructed by Theorem
\ref{ext:thm_xi} have associated matrix $M=\matrice{3}{1}{-1}{0}$,
and so a short computation gives:

\begin{lemma}
 \label{conj:lemma:xi_um}
For each $\um\in\Sigma^*$, the conjugates of $\xi_\um$ are
$\xi_\um-3$ and $\xi_\um+3$ and its associated quadratic forms are,
up to sign,
\[
 G_\um(U,T):=(T-\xi_\um U)(T-(\xi_\um+3)U)
 \et
 G_\um(U,T+3U).
\]
\end{lemma}

In the computations below, we use the fact that, for any
$A,B\in\cP$, the integer $c$ determined by $A*B = c^{-1}AB$ is a
common divisor of $\det(A)$ and $\det(B)$.  We also use the estimate
$X_{i+1}\asymp X_i^\gamma$ coming from \eqref{ext:propExt:eq1}. The
next lemma relates the forms $F_i$ and $G_i$.

\begin{lemma}
 \label{conj:lemmaFG}
For each $i\ge 1$, there exists a non-zero rational number $r_i$
with $|r_i|\asymp X_i$ such that $F_i = r_i G_i + \cO(X_i^{-1})$.
\end{lemma}

\begin{proof}
Since $W_i=\ux_i*M_i$, we have $W_i=c_i^{-1}\ux_iM_i$ for some
divisor $c_i$ of $\det(M)$.  Therefore, for $i$ large enough, the
rational number $r_i = x_{i,0}/c_i$ is non-zero and satisfies $|r_i|
\asymp |x_{i,0}| \asymp X_i$ as well as
\[
 \| F_i - r_i G_i \|
 \ll
 \Big\| \ux_i - x_{i,0}\matrice{1}{\xi}{\xi}{\xi^2} \Big\|
 \asymp
 \| (\xi,-1)\ux_i \|
 \asymp X_i^{-1}.
\]
\end{proof}

The next result provides an alternative formula for the forms $F_i$
showing that they are essentially homogenous versions of the
quadratic polynomials of \cite[\S8]{RcubicI}.

\begin{lemma}
 \label{conj:lemmaF}
For each $i\ge 1$, we have
\begin{equation}
 \label{conj:lemmaF:eq1}
 F_i(U,T)
 = \frac{1}{d_i}
   \left| \begin{matrix}
   U^2 &UT &T^2\\
   x_{i+1,0} &x_{i+1,1} &x_{i+1,2}\\
   x_{i+2,0} &x_{i+2,1} &x_{i+2,2}
   \end{matrix} \right|
\end{equation}
where $d_i$ is a divisor of $\det(\ux_{i+1})$.  Moreover the content
of $F_i$ as a polynomial in $\bZ[U,T]$ is bounded above
independently of $i$.
\end{lemma}

\begin{proof}
Thanks to the formulas of \cite[\S2]{RcubicI}, the determinant in
the right hand side of \eqref{conj:lemmaF:eq1} can be rewritten as
\[
 \trace\Big( \begin{pmatrix}U^2 &UT\\ UT &T^2\end{pmatrix}
               J \ux_{i+2} J \ux_{i+1} J \Big),
\]
where the symbol $\trace$ stands for the trace.  Since $\ux_{i+2} =
W_i*\ux_{i+1} = \kappa_i^{-1} W_i\ux_{i+1}$ for some divisor
$\kappa_i$ of $\det(\ux_{i+1})$ and since $\ux_{i+1}J\ux_{i+1}J =
-\det(\ux_{i+1})I$, this expression becomes
\[
  - \frac{\det(\ux_{i+1})}{\kappa_i}
    \trace\Big( \begin{pmatrix}U^2 &UT\\ UT &T^2\end{pmatrix}
               J W_i \Big)
 = \frac{\det(\ux_{i+1})}{\kappa_i} F_i(U,T).
\]
This proves the first assertion. Identifying any symmetric matrix
$\matrice{m}{k}{k}{\ell}$ with the triple $(m,k,\ell)$, the formula
\eqref{conj:lemmaF:eq1} implies that the content of $F_i$ divides
$\det(\ux_i,\ux_{i+1},\ux_{i+2})$.  The second assertion follows
since, by \cite[Thm~5.1]{RcubicI}, the absolute value of this
determinant is bounded above independently of $i$.
\end{proof}

Combining the above lemma with the results of \cite[\S8]{RcubicI},
we obtain:

\begin{proposition}
 \label{conj:prop:alpha}
There exists an integer $i_0\ge 1$ such that, for each $i\ge i_0$,
the polynomial $F_i(U,T)$ is irreducible over $\bQ$ and the root
$\alpha_i$ of $F_i(1,T)$ which is closest to $\xi$ is algebraic over
$\bQ$ of degree $2$ with
\[
  H(\alpha_i) \asymp \|F_i\| \asymp X_i
  \et
  |\xi-\alpha_i| \asymp H(\alpha_i)^{-2\gamma-2}.
\]
Moreover, for each algebraic number $\alpha\in\bC$ of degree $\le 2$
over $\bQ$ with $\alpha\neq \alpha_i$ for each $i\ge i_0$, we have
$|\xi-\alpha| \gg H(\alpha)^{-4}$.
\end{proposition}

\begin{proof}
According to \cite[Thm.~8.2]{RcubicI}, the polynomial
$Q_{i+1}(T):=d_iF_i(1,T)$ is irreducible over $\bQ$ for each
sufficiently large $i$.  For those $i$, the quadratic form
$F_i(U,T)$ is irreducible over $\bQ$ and $\alpha_i$ is algebraic
over $\bQ$ of degree $2$.  Moreover, since by Lemma
\ref{conj:lemmaF} the integer $d_i$ and the content of $F_i$ are
bounded, we deduce that $H(\alpha_i) \asymp \|F_i\| \asymp
\|Q_{i+1}\|$.  According to \cite[Prop.~8.1]{RcubicI}, we also have
$\|Q_{i+1}\| \asymp X_i$.  The remaining estimates follow from
\cite[Thm.~8.2]{RcubicI}.
\end{proof}

\begin{definition}
 \label{conj:def:best_app}
In view of the above proposition, the sequence $(\alpha_i)_{i\ge
i_0}$ is uniquely determined by the extremal number $\xi$ up to its
first terms. We refer to it as a sequence of \emph{best quadratic
approximations} to $\xi$.
\end{definition}

The next lemma provides such sequences for the extremal numbers
$\xi_\um$ defined in Theorem \ref{ext:thm_xi}, in terms of the
quadratic numbers $\alpha_\um$ given by \eqref{M:alpha_m}.

\begin{lemma}
 \label{conj:lemma:best_app}
Let $\um\in\Sigma^*$ and let $(\um^{(i)})_{i\ge 1}$ denote the
maximal zigzag in the tree \eqref{M:Mtreex} starting with
$\um^{(1)}=\um$.  Put $r=1$ if $\um^{(2)}$ is the right successor of
$\um^{(1)}$ and $r=0$ otherwise.  Then a sequence $(\alpha_i)_{i\ge
1}$ of best quadratic approximations to $\xi_\um$ is given by
\begin{equation}
 \label{conj:lemma:best_app:eq}
 \alpha_i
  = \begin{cases}
    \alpha_{\um^{(i)}} &\text{if\quad $i\equiv r \mod 2$,}\\
    \alphabar_{\um^{(i)}}+3 &\text{if\quad $i\not\equiv r \mod 2$.}
    \end{cases}
\end{equation}
\end{lemma}

\begin{proof}
Define $\ux_i = \ux_{\um^{(i)}}$ for each $i\ge 1$ so that
$(\ux_i)_{i\ge 1}$ is a sequence of symmetric matrices in
$\SL_2(\bZ)$ corresponding to $\xi_\um$ (see Theorem
\ref{ext:thm_xi}).  By virtue of the choice of $r$, the triple
$\um^{(i+1)}$ is a right successor of $\um^{(i)}$ in
\eqref{M:Mtreex} if and only if $i\equiv r \mod 2$.  From this we
deduce that
\[
 \ux_{i+2}
 = \ux_{i+1} \matrice{3}{(-1)^{i-r+1}}{(-1)^{i-r}}{0} \ux_i
\]
for each $i\ge 1$ (same argument as in the first paragraph of the
proof of Theorem \ref{ext:thm_xi}).  Thus, in view of Proposition
\ref{conj:prop:alpha}, it remains simply to show that, for each
sufficiently large $i$, the real number defined by
\eqref{conj:lemma:best_app:eq} is the root of the polynomial
\[
 -\begin{pmatrix}1 &T \end{pmatrix} J \ux_i
  \matrice{3}{(-1)^{i-r}}{(-1)^{i-r-1}}{0}
  \begin{pmatrix}1\\ T \end{pmatrix}
\]
which is closest to $\xi_\um$.  If $i\equiv r$ mod $2$, this
polynomial is simply $F_{\um^{(i)}}(1,T)$ (with the notation of
\eqref{M:F_um}). If $i\not\equiv r$ mod $2$, a short computation
shows that it is equal to $-F_{\um^{(i)}}(1,T-3)$.  The conclusion
follows since the roots of $F_{\um^{(i)}}(1,T)$ are
$\alpha_{\um^{(i)}}$ and $\alphabar_{\um^{(i)}}$ which, according to
\eqref{ext:thm_xi:lim}, converge respectively to $\xi_\um$ and
$\xi_\um-3$ as $i\to\infty$.
\end{proof}

The next result justifies the terminology of Definition
\ref{conj:def}.

\begin{proposition}
 \label{conj:prop:xi'xi''}
Let $(\alpha_i)_{i\ge i_0}$ be as in Proposition
\ref{conj:prop:alpha}.  Then, as $i\to\infty$, we have
\begin{equation}
 \label{conj:prop:xi'xi'':estimates}
  |\xi' - \alphabar_{2i-1}| \asymp H(\alpha_{2i-1})^{-2}
  \et
  |\xi''- \alphabar_{2i}| \asymp H(\alpha_{2i})^{-2}.
\end{equation}
Therefore, the sequence of conjugates of a sequence of best
quadratic approximations to $\xi$ admits exactly two accumulation
points, namely the conjugates $\xi'$ and $\xi''$ of $\xi$.
\end{proposition}

\begin{proof}
We simply prove \eqref{conj:prop:xi'xi'':estimates} since the second
assertion follows from it. For each $i\ge i_0$, let $p_i :=
F_i(0,1)$ denote the coefficient of $T^2$ in $F_i(U,T)$. If $i\ge
i_0$ is odd, Lemma \ref{conj:lemmaFG} gives
\[
 (T - \alpha_i U)(T - \alphabar_i U)
  = p_i^{-1}F_i(U,T)
  = (T-\xi U)(T-\xi'U) + \cO(X_i^{-2}),
\]
and therefore $\alpha_i+\alphabar_i = \xi+\xi'+\cO(X_i^{-2})$ by
comparing the coefficients of $UT$.  Since Proposition
\ref{conj:prop:alpha} gives $\|F_i\|\asymp X_i$ and $|\alpha_i-\xi|
\asymp X_i^{-2\gamma-2}$, we deduce that $|p_i|\asymp X_i$ and
$|\alphabar_i - \xi'| \ll X_i^{-2}$.  To bound $|\alphabar_i -
\xi'|$ from below, we first note that, since $\xi'\neq \xi$, the
above estimates imply
\[
 |\alpha_i-\alpha_{i+2}|\asymp  X_i^{-2\gamma-2},\quad
 |\alphabar_i-\alpha_{i+2}|\asymp  1,\quad
 |\alpha_i-\alphabar_{i+2}|\asymp  1,
\]
and so the resultant of $F_i$ and $F_{i+2}$ satisfies
\[
 \begin{aligned}
 |\Res(F_i,F_{i+2})|
 &= p_i^2p_{i+2}^2\, |\alpha_i-\alpha_{i+2}|\,
 |\alphabar_i-\alpha_{i+2}|\, |\alpha_i-\alphabar_{i+2}|\,
 |\alphabar_i-\alphabar_{i+2}|\\
 &\ll X_i^2 X_{i+2}^2 X_i^{-2\gamma-2}
     \big(|\alphabar_i - \xi'| + \cO(X_{i+2}^{-2})\big)\\
 &\ll X_i^2 |\alphabar_i - \xi'| + \cO(X_{i+1}^{-2}).
 \end{aligned}
\]
If $i$ is large enough this resultant is a non-zero integer.  Its
absolute value is then bounded below by $1$, and the above estimate
leads to $|\alphabar_i - \xi'| \gg X_i^{-2}$, thus $|\alphabar_i -
\xi'| \asymp X_i^{-2} \asymp H(\alpha_i)^{-2}$. The proof for $i$
even is similar: it suffices to replace everywhere $\xi'$ by
$\xi''$.
\end{proof}

\begin{corollary}
 \label{conj:cor:A.xi}
For each $A\in\GL_2(\bQ)$, the conjugates of $A\cdot\xi$ are
$A\cdot\xi'$ and $A\cdot\xi''$.
\end{corollary}

\begin{proof}
Fix $A\in\GL_2(\bQ)$ and a sequence $(\alpha_i)_{i\ge 1}$ of best
quadratic approximations to $\xi$. Since
\[
 | A\cdot\xi - A\cdot\alpha_i |
   \asymp |\xi-\alpha_i|
   \asymp H(\alpha_i)^{-2\gamma-2}
   \asymp H(A\cdot\alpha_i)^{-2\gamma-2},
\]
we deduce that $(A\cdot\alpha_i)_{i\ge 1}$ is a sequence of best
quadratic approximations to the extremal number $A\cdot\xi$.  Thus
the conjugates of $A\cdot\xi$ are the accumulation points of the
sequence $(A\cdot\alphabar_i)_{i\ge 1}$, namely $A\cdot\xi'$ and
$A\cdot\xi''$.
\end{proof}

Based on this proposition a simple computation gives:

\begin{corollary}
 \label{conj:cor:N.xi}
Let $\disp N = \begin{pmatrix}b &a\\ -d &-c\end{pmatrix}$.  Then we
have $\xi'=N\cdot\xi$ and $\xi''=N^{-1}\cdot\xi$.  Moreover, for
each $i\in \bZ$, the conjugates of $N^i\cdot\xi$ are
$N^{i-1}\cdot\xi$ and $N^{i+1}\cdot\xi$.
\end{corollary}

In particular, this shows that $\xi$ is one of the two conjugates of
$\xi'$ and also one of the two conjugates of $\xi''$.  Although we
will not need the next result in the sequel, we decided to include
it as it provides an attractive complement to Proposition
\ref{conj:prop:alpha}.

\begin{theorem}
 \label{conj:thm}
Let $(\alpha_i)_{i\ge i_0}$ be as in Proposition
\ref{conj:prop:alpha}. For each $i\ge i_0$, define
\[
 \alpha_i'
 =
 \begin{cases}
  \alpha_i &\text{if $i$ is odd,}\\
  N\cdot\alphabar_i &\text{if $i$ is even,}
 \end{cases}
\]
where $N$ is the integral matrix of Corollary \ref{conj:cor:N.xi},
then,
\begin{equation}
 \label{conj:thm:eq1}
 |\xi-\alpha_i'|\,|\xi'-\alphabar_i'|
 \asymp
 H(\alpha_i')^{-2\gamma-4}.
\end{equation}
For each quadratic or rational number $\alpha\in\bC$ not belonging
to the sequence $(\alpha_i')_{i\ge i_0}$, we have instead
\begin{equation}
\label{conj:thm:eq2}
 |\xi-\alpha|\,|\xi'-\alphabar| \gg H(\alpha)^{-6}
\end{equation}
where $\alphabar$ denotes the conjugate of\/ $\alpha$ over $\bQ$.
\end{theorem}

\begin{proof}
If $i$ is odd, the estimate \eqref{conj:thm:eq1} follows from
Propositions \ref{conj:prop:alpha} and \ref{conj:prop:xi'xi''} since
$\alpha_i' = \alpha_i$ and $\alphabar_i'=\alphabar_i$.  If $i$ is
even, we find
\[
 \begin{aligned}
 |\xi-\alpha_i'|\,|\xi'-\alphabar_i'|
 &= |\xi-N\cdot\alphabar_i|\,|\xi'-N\cdot\alpha_i| \\
 &\asymp |N^{-1}\cdot\xi-\alphabar_i|\,|N^{-1}\cdot\xi'-\alpha_i|
 = |\xi''-\alphabar_i|\,|\xi-\alpha_i|
 \end{aligned}
\]
and \eqref{conj:thm:eq1} again follows from Propositions
\ref{conj:prop:alpha} and \ref{conj:prop:xi'xi''} because
$H(\alpha_i) \asymp H(\alpha_i')$.

To prove the second part of the theorem, we first note that, if
$\alpha = \alpha_i$ for some even integer $i$, then Proposition
\ref{conj:prop:alpha} provides $|\xi-\alpha| \asymp
H(\alpha)^{-2\gamma-2}$ while the estimates of Proposition
\ref{conj:prop:xi'xi''} lead to $|\xi' - \alphabar| \asymp 1$ since
$\xi'\neq \xi''$. Similarly, if $\alpha = N \cdot \alphabar_i$ for
some odd integer $i$, we find $|\xi' - \alphabar| = |N\cdot\xi -
N\cdot\alpha_i| \asymp |\xi - \alpha_i|\asymp
H(\alpha)^{-2\gamma-2}$ and $|\xi-\alpha| \asymp |N^{-1}\cdot\xi -
\alphabar_i| = |\xi'' - \alphabar_i| \asymp 1$. In both cases, this
leads to
\[
 |\xi-\alpha|\,|\xi'-\alphabar|
   \asymp H(\alpha)^{-2\gamma-2}
   \gg H(\alpha)^{-6}.
\]
If $\alpha = \alphabar_i$ for any integer $i\ge i_0$, then we find
instead $|\xi-\alpha| \asymp |\xi'-\alphabar| \asymp 1$ and so
\eqref{conj:thm:eq2} holds again.  The same estimate holds if
$\alpha \in \bQ$ because in that case we have $|\xi-\alpha|\gg
H(\alpha)^{-3}$ and $|\xi'-\alpha|\gg H(\alpha)^{-3}$ by \cite[Thm
1.3]{RcubicI}.  We may therefore assume that $\alpha$ is irrational
and different from $\alpha_i$, $\alphabar_i$ and $N\cdot\alphabar_i$
for each $i\ge i_0$. In this case, Proposition \ref{conj:prop:alpha}
gives
\begin{equation}
\label{conj:thm:eq3}
 |\xi-\alpha|\gg H(\alpha)^{-4}
 \et
 |\xi'-\alphabar| \asymp |\xi-N^{-1}\cdot\alphabar|
                  \gg H(\alpha)^{-4}.
\end{equation}
Let $p$ denote the positive integer for which the polynomial
\[
 F(U,T) := p (T-\alpha U)(T-\alphabar U)
\]
has relatively prime integer coefficients.  Then, $F$ is an
irreducible polynomial of $\bZ[T]$ and, for each $i\ge i_0$, we have
\[
 1 \le |\Res(F,F_i)|
    = p^2 p_i^2 |\alpha-\alpha_i|\, |\alpha-\alphabar_i|\,
    |\alphabar-\alpha_i|\, |\alphabar-\alphabar_i|,
\]
where $p_i = F_i(0,1)$.  Since
\[
\begin{aligned}
 p\, |p_i|\, |\alpha-\alphabar_i|\, |\alphabar-\alpha_i|
 &\le p\, |p_i|
     \big( 2 \max\{1,|\alpha|\} \max\{1,|\alphabar_i|\}\big)
     \big( 2 \max\{1,|\alphabar|\} \max\{1,|\alpha_i|\}\big)\\
 &= 4 \big( p \max\{1,|\alpha|\} \max\{1,|\alphabar|\}\big)
     \big( |p_i| \max\{1,|\alpha_i|\} \max\{1,|\alphabar_i|\}\big)\\
 &\ll H(\alpha)H(\alpha_i),
 \end{aligned}
\]
we deduce that
\[
 1 \ll H(\alpha)^2 H(\alpha_i)^2 |\alpha-\alpha_i|\,
      |\alphabar-\alphabar_i|.
\]
If $i$ is odd, Propositions \ref{conj:prop:alpha} and
\ref{conj:prop:xi'xi''} also give $H(\alpha_i)\asymp X_i$,
$|\alpha-\alpha_i| \le |\xi-\alpha| + \cO(X_i^{-2\gamma-2})$ and
$|\alphabar-\alphabar_i| \le |\xi'-\alphabar| + \cO(X_i^{-2})$.
Combining these estimates, we deduce the existence of a constant
$c>0$ such that
\[
 c \le H(\alpha)^2 X_i^2 \big( |\xi-\alpha| + X_i^{-2\gamma-2} \big)
      \big( |\xi'-\alphabar| + X_i^{-2} \big),
\]
for each odd integer $i$.  If $|\xi-\alpha| \ge (c/4)H(\alpha)^{-2}$
or $|\xi'-\alphabar| \ge (c/4)H(\alpha)^{-2}$, then the required
estimate \eqref{conj:thm:eq2} follows from \eqref{conj:thm:eq3} and
we are done.  Otherwise, we obtain
\[
 \frac{c}{2}
   \le H(\alpha)^2 X_i^{-2\gamma-2}
       + H(\alpha)^2 X_i^2\, |\xi-\alpha|\, |\xi'-\alphabar|.
\]
Choose $i$ to be the smallest positive odd integer such that
$H(\alpha)^2 X_i^{-2\gamma-2} \le c/4$.  Then we have $X_i \ll
H(\alpha)^{1/\gamma}$ and we obtain
\[
 \frac{c}{4}
   \le H(\alpha)^2 X_i^2\, |\xi-\alpha|\, |\xi'-\alphabar|
   \ll H(\alpha)^{2\gamma} |\xi-\alpha|\, |\xi'-\alphabar|,
\]
which is stronger than \eqref{conj:thm:eq2}.
\end{proof}

\begin{remark}
A similar argument shows that Theorem \ref{conj:thm} holds with
$\xi'$ replaced by $\xi''$ and $\alpha'_i$ replaced by $\alpha_i''$
where $\alpha_i'' = \alpha_i$ if $i$ is even, and $\alpha_i'' =
N^{-1}\cdot\alphabar_i$ if $i$ is odd.
\end{remark}

%
%

\section{Minima of the associated real quadratic forms}
\label{sec:reduction}

We keep the notation of the preceding section.  In particular we
deal with a fixed arbitrary extremal number $\xi$ with conjugates
$\xi'$ and $\xi''$ and associated quadratic forms $G'$ and $G''$.
The main result of this section is that $\nu(\xi) =
\mu(G')/\sqrt{\disc(G')} = \mu(G'')/\sqrt{\disc(G'')}$.  We will
deduce from this that the extremal numbers $\xi_\um$
($\um\in\Sigma^*$) constructed by Theorem \ref{ext:thm_xi} have
Lagrange constant $\nu(\xi_\um) = 1/3$.  The proof goes through a
series of lemmas.

\begin{lemma}
 \label{red:lemmad}
Let $d$ denote the least common multiple of all integers $\det(W_i)$
with $i\ge 1$.  Suppose that $W_i\equiv W_j \mod 4d$ for some
indices $i,j\ge 1$.  Then, we have $W_jW_i^{-1} \in\SL_2(\bZ)$.
\end{lemma}

\begin{proof}
This follows from the formula $W_jW_i^{-1} = \det(W_i)^{-1} W_j
\adj(W_i)$ where $\adj(W_i)$ denotes the adjoint of $W_i$.  Since
$W_j\adj(W_i) \equiv W_i\adj(W_i) \equiv \det(W_i) I \mod 4d$, and
since $\det(W_i)$ divides $d$, the matrix $W_jW_i^{-1}$ has integer
coefficients.  Moreover, as $\det(W_i)$ and $\det(W_j)$ divide $d$
and are congruent modulo $4d$, they must be equal, and so
$\det(W_jW_i^{-1})=1$.
\end{proof}

\begin{lemma}
 \label{red:lemmak}
Let $i_0\in\{0,1\}$.  There exists an integer $k\ge 1$ such that
$W_{i+2k}W_i^{-1}\in \SL_2(\bZ)$ for an infinite set of indices
$i\ge 1$ with $i\equiv i_0 \mod 2$.
\end{lemma}

\begin{proof} Let $d$ be as in Lemma \ref{red:lemmad}, and let
$N=(4d)^4$ denote the number of congruence classes of $2\times 2$
integral matrices modulo $4d$.  For each integer $j\ge 1$ with
$j\equiv i_0 \mod 2$, at least two matrices among $W_j, W_{j+2},
\dots, W_{j+2N}$ are congruent modulo $4d$.  So there exist integers
$i$ and $k$ with $i\ge j$, $i\equiv i_0 \mod 2$ and $1\le k\le N$
such that $W_i\equiv W_{i+2k} \mod 4d$.  By varying $j$, we get
infinitely many such pairs $(i,k)$.  As $k$ stays within a finite
set, at least one value of $k$ arises infinitely many often. The
conclusion follows by Lemma \ref{red:lemmad}.
\end{proof}

\begin{lemma}
 \label{red:lemmaWWW}
For each $i\ge 2$, we have $\|W_iW_{i-1}W_i-W_{i-1}W_i^2\| \asymp
X_{i-1}$.
\end{lemma}

\begin{proof}
Since $W_i = \ux_i*M_i$ and $W_{i-1} = \ux_{i-1}*M_{i-1}$ are
respectively quotients of $\ux_iM_i$ and $\ux_{i-1}M_{i-1}$ by
divisors of $\det(M)$, this amounts to showing that
\[
 \|(\ux_iM_i\ux_{i-1}M_{i-1}-\ux_{i-1}M_{i-1}\ux_iM_i)\ux_iM_i\|
  \asymp
 X_{i-1}.
\]
Since $\ux_iM_i\ux_{i-1} = \ux_{i-1}M_{i-1}\ux_i$ is the product of
$\ux_{i+1}$ by a divisor $\kappa$ of $\det(\ux_i)\det(M)$ and since
the latter is a bounded integer, this in turn amounts to showing
that
\[
 \|\ux_{i+1} (M_{i-1} - M_i) \ux_iM_i\|  \asymp X_{i-1}.
\]
Finally, since $M_{i-1} - M_i = \pm (M-\tM) = \pm (b-c)J$, this last
estimate follows from the fact that $\ux_{i+1}J\ux_i = \kappa^{-1}
\ux_{i-1}M_{i-1}\ux_i J \ux_i = \kappa^{-1} \det(\ux_i)
\ux_{i-1}M_{i-1} J$ has norm of the same order as $\|\ux_{i-1}\| =
X_{i-1}$.
\end{proof}

\begin{lemma}
 \label{red:lemmaF}
For each $i\ge 2$, we have
\begin{equation}
 \label{red:lemmaF:eq0}
 \|F_{i+2}((U,T)\tW_i)-\det(W_i)F_{i+2}(U,T)\| \ll X_{i-1}.
\end{equation}
\end{lemma}

\begin{proof}
The left hand side of \eqref{red:lemmaF:eq0} is the norm of the
polynomial $\begin{pmatrix} U &T\end{pmatrix} A \begin{pmatrix} U\\
T\end{pmatrix}$ where
\[
 -A = \tW_i J W_{i+2}W_i -\det(W_i)J W_{i+2}.
\]
Since $W_{i+2} = W_{i+1}*W_i = W_i*W_{i-1}*W_i$, we find that
\[
\begin{aligned}
 \|A\|
 &\asymp \|\tW_i JW_iW_{i-1}W_i^2 - \det(W_i) JW_iW_{i-1}W_i \| \\
 &= |\det W_i|\,\|JW_{i-1}W_i^2 - JW_iW_{i-1}W_i \| \\
 &\ll X_{i-1},
\end{aligned}
\]
where the last estimate comes from Lemma \ref{red:lemmaWWW}.  The
conclusion follows.
\end{proof}

\begin{lemma}
 \label{red:lemmaGW}
For each $i\ge 2$, we have
\begin{equation*}
 \|G_i((U,T)\tW_i)-\det(W_i)G_i(U,T)\| \ll X_i^{-2}.
\end{equation*}
\end{lemma}

\begin{proof}
Since $G_i=G_{i+2}$, Lemma \ref{conj:lemmaFG} shows that $G_i =
r_{i+2}^{-1}F_{i+2} +\cO(X_{i+2}^{-2})$ for some non-zero rational
number $r_{i+2}$ with $|r_{i+2}|\asymp X_{i+2}$. As
$\|W_i\|\asymp X_i$, this gives
\[
 \begin{aligned}
 G_i((U,T)\tW_i)
 &= r_{i+2}^{-1}F_{i+2}((U,T)\tW_i) +\cO(X_i^2X_{i+2}^{-2}) \\
 &= r_{i+2}^{-1}\det(W_i)F_{i+2}(U,T) +\cO(X_{i+2}^{-1}X_{i-1})
   \quad\text{by Lemma \ref{red:lemmaF},}\\
 &= \det(W_i)G_i(U,T) +\cO(X_{i+2}^{-1}X_{i-1}).
 \end{aligned}
\]
\end{proof}

\begin{lemma}
 \label{red:lemmaGSik}
For any integers $i\ge 1$ and $k\ge 0$, the matrix $S_{i,k} :=
W_{i+2k}W_i^{-1}$ satisfies
\begin{equation*}
 \|G_{i+1}((U,T)\tS_{i,k})-\det(S_{i,k})G_{i+1}(U,T)\|
 \le cX_{i+1}^{-2},
\end{equation*}
with a constant $c>0$ which is independent of both $i$ and $k$.
\end{lemma}

\begin{proof}
Define $H_{i,k}(U,T) = G_{i+1}((U,T)\tS_{i,k}) - \det(S_{i,k})
G_{i+1}(U,T)$ for each $i\ge 1$ and $k\ge 0$.  When $k\ge 1$, we
have
\[
 S_{i,k} =  S_{i+2,k-1}W_{i+2}W_i^{-1}
         = a_i^{-1} S_{i+2,k-1}W_{i+1}
\]
for some bounded positive integer $a_i$, and so
\[
 \begin{aligned}
 H_{i,k}(U,T)
  &= a_i^{-2} H_{i+2,k-1}((U,T)\tW_{i+1}) \\
  &\quad + a_i^{-2} \det(S_{i+2,k-1})
      \big( G_{i+1}((U,T)\tW_{i+1}) - \det(W_{i+1})G_{i+1}(U,T)
      \big).
 \end{aligned}
\]
Since $|\det(S_{i+2,k-1})| \le |\det(W_{i+2})| \ll 1$, we deduce
from Lemma \ref{red:lemmaGW} that
\begin{equation}
\label{red:lemmaGSik:eq1}
 \|H_{i,k}\|
  \le c_1 \|H_{i+2,k-1}\|\,X_{i+1}^2 + c_1 X_{i+1}^{-2}
\end{equation}
with a constant $c_1>0$ which is independent of $i$ and $k$.  Put
$h_{i,k}=\|H_{i,k}\|X_{i+1}^2$ and choose $c_2>0$ such that
$X_iX_{i+1} \le c_2X_{i+2}$ for each $i\ge 1$.  Then, we find
$X_{i+3}^{-2} \le c_2^4 X_i^{-2} X_{i+1}^{-4}$ and so
\eqref{red:lemmaGSik:eq1} leads to
\begin{equation}
\label{red:lemmaGSik:eq2}
 h_{i,k} \le c_1 + c_1c_2^4 X_i^{-2} h_{i+2,k-1},
\end{equation}
for any $i,k\ge 1$.  Our goal is to show that $h_{i,k}$ is bounded
above independently of $i$ and $k$.  To this end, we choose an
integer $i_0\ge 1$ such that $X_i^2\ge 2c_1c_2^4$ for each $i\ge
2i_0$.  Then \eqref{red:lemmaGSik:eq2} gives  $h_{i,k} \le c_1 +
(1/2) h_{i+2,k-1}$ for each $i\ge 2i_0$ and $k\ge 1$.  Since
$h_{i+2k,0}=0$, this implies that $h_{i,k}\le 2c_1$ whenever $i\ge
2i_0$.  If $1\le i < 2i_0 \le 2k$, the estimate
\eqref{red:lemmaGSik:eq2} leads to $h_{i,k} \ll 1 + h_{i+2i_0,k-i_0}
\le 1+2c_1$.  We conclude that $h_{i,k}\ll 1$ for any $i\ge 1$ and
$k\ge 0$.
\end{proof}

\begin{lemma}
 \label{red:lemmaGS}
Let $G$ stand for one of the polynomials $G'$ or $G''$.  For each
$\delta>0$, there exists a matrix $S\in\SL_2(\bZ)$ which satisfies
both
\begin{equation}
 \label{red:lemmaGS:eq0}
 \|(\xi,-1)S\| \le \delta \et \|G((U,T)\tS)-G(U,T)\| \le \delta.
\end{equation}
\end{lemma}

\begin{proof}
Put $i_0=0$ if $G=G'$ and $i_0=1$ if $G=G''$, so that $G=G_{i+1}$
for each integer $i\ge 1$ with $i\equiv i_0 \mod 2$.  By Lemma
\ref{red:lemmak}, there exists an integer $k\ge 1$ such that
$S_{i,k}=W_{i+2k}W_i^{-1} \in \SL_2(\bZ)$ for an infinite set $I$ of
positive integers $i$ with $i\equiv i_0 \mod 2$.  Since $W_i^{-1} =
\det(W_i)^{-1} \adj(W_i)$ and $W_{i+2k}= \ux_{i+2k}*M_i$, we find
that
\[
 \|(\xi,-1)S_{i,k}\|
 \ll \|(\xi,-1)\ux_{i+2k}\|\,\|W_i\|
 \ll X_{i+2k}^{-1}X_i \ll X_{i+1}^{-1}.
\]
This combined with Lemma \ref{red:lemmaGSik} shows that, given
$\delta>0$, the matrix $S=S_{i,k}$ satisfies \eqref{red:lemmaGS:eq0}
for each sufficiently large $i\in I$.
\end{proof}

\begin{theorem}
 \label{red:thm}
We have  $\disp \nu(\xi) = \frac{\mu(G')}{\sqrt{\disc(G')}} =
\frac{\mu(G'')}{\sqrt{\disc(G'')}}$.
\end{theorem}

\begin{proof}
We have $\disc(G') = \theta^2$ where $\theta:=(c+d\xi)(\xi-\xi')$,
and
\begin{equation}
 \label{red:thm:eq1}
 \begin{aligned}
 G'(U,T) &= (c+d\xi)(T-\xi U)(T-\xi'U)
         = (c+d\xi)(T-\xi U)^2 + \theta (T-\xi U)U.
 \end{aligned}
\end{equation}
Fix a real $\epsilon$ with $0<\epsilon <1$.  By definition, there
exists a non-zero point $(u,t)\in\bZ^2$ for which $|G'(u,t)| \le
\mu(G')+\epsilon$.  Then, by Lemma \ref{red:lemmaGS}, there exists
$S \in \SL_2(\bZ)$ such that the point $(q,p) =(u,t)\tS \in \bZ^2$
satisfies both $|q\xi-p| = \Big|(\xi,-1)S \begin{pmatrix} u\\t
\end{pmatrix}\Big| \le \epsilon$ and $|G'(q,p)-G'(u,t)| \le
\epsilon$.  Combining this with \eqref{red:thm:eq1}, we deduce that
\[
 \mu(G') +2\epsilon
 \ge
 |G'(q,p)|
 \ge
 |\theta|\,|q(q\xi-p)| - |c+d\xi|\epsilon^2.
\]
By letting $\epsilon$ tend to $0$, the integer $|q|$ tends to
infinity and we conclude that $\mu(G') \ge |\theta|\nu(\xi)$.

The reverse inequality follows directly from \eqref{red:thm:eq1} by
observing that, for each $\epsilon>0$, there exists a point
$(q,p)\in\bZ^2$ with $q\ge 1$, $|q\xi-p|\le \epsilon$ and $q|q\xi-p|
\le \nu(\xi)+\epsilon$ and so by \eqref{red:thm:eq1} we obtain
$\mu(G') \le |G'(q,p)| \le |\theta|(\nu(\xi)+\epsilon) +
|c+d\xi|\epsilon^2$ which upon letting $\epsilon \to 0$ gives
$\mu(G') \le |\theta|\nu(\xi)$.  This shows that $\mu(G') =
\sqrt{\disc(G')}\,\nu(\xi)$. The proof for $G''$ is similar.
\end{proof}

\begin{corollary}
 \label{red:thm:cor1}
 We have $\disp \nu(\xi) = \nu(\xi') = \nu(\xi'')$.
\end{corollary}

\begin{proof}
By Corollary \ref{conj:cor:N.xi}, $\xi$ is one of the two conjugates
of $\xi'$. Thus, $G'$ is also one of the two real quadratic
polynomials associated to $\xi'$ and so Theorem \ref{red:thm} gives
$\nu(\xi') = \mu(G')/\sqrt{\disc(G')} = \nu(\xi)$.  Similarly, we
find that $\nu(\xi'')=\nu(\xi)$.
\end{proof}

\begin{corollary}
 \label{red:thm:cor2}
For any $\um\in\Sigma^*$, we have $\disp \nu(\xi_\um) = \mu(G_\um)/3
= 1/3$ where $G_\um$ is as in Lemma \ref{conj:lemma:xi_um}.
\end{corollary}

\begin{proof}
Fix $\um\in \Sigma^*$.  By Theorem \ref{red:thm}, we have
$\nu(\xi_\um) = \mu(G_\um)/3$ since $\disc(G_\um)=9$.  According to
Theorem \ref{ext:thm_xi}, we also have $\xi_\um = \lim_{i\to\infty}
\alpha_{\um^{(i)}} = \lim_{i\to\infty} (\alphabar_{\um^{(i)}}+3)$
where $(\um^{(i)})_{i\ge 1}$ denote the maximal zigzag in the tree
\eqref{M:Mtreex} originating from $\um$.  In terms of the quadratic
forms \eqref{M:F_um}, this means that
\[
 \frac{G_\um}{3}
  = \lim_{i\to \infty}
    \frac{F_{\um^{(i)}}}{\sqrt{\disc(F_{\um^{(i)}})}}
\]
and thus $\mu(G_\um)/3 \ge \limsup_{i\to\infty}
\mu(F_{\um^{(i)}})/\sqrt{\disc(F_{\um^{(i)}})}$.  Finally, Theorem
\ref{M:thmM} shows that the latter limit superior is equal to $1/3$.
This gives $\nu(\xi_\um)\ge 1/3$ and, since $\xi_\um$ is not
quadratic, we conclude that $\nu(\xi_\um)=1/3$.
\end{proof}

%
%

\section{Continued fraction expansions}
 \label{sec:cont}

In this section we define notions of \emph{reduced} and
\emph{balanced} extremal numbers and we describe the continued
fraction expansions of the extremal numbers $\xi_\um$ introduced in
\S\ref{sec:ext}.  To begin, we first set additional notation and
recall some basic facts about continued fraction expansions.

Let $\cW$ denote the monoid of words on the set $\{1,2,3,\dots\}$ of
positive integers with the product given by concatenation of words.
For any non-empty word $\uw$ of $\cW$ written either as a sequence
$\uw=(a_1,\dots,a_k)$ or as a string $\uw=a_1\cdots a_k$, we define
\begin{equation*}
 \varphi(\uw)
 = \matrice{a_1}{1}{1}{0} \cdots \matrice{a_k}{1}{1}{0}
 \in \GL_2(\bZ),
\end{equation*}
and for the empty word $\emptyset$, we set $\varphi(\emptyset)=I$.
Then the map $\varphi\colon\cW\to \GL_2(\bZ)$ is a morphism of
monoids and, with our convention that the norm of a matrix is the
maximum of the absolute values of its coefficients, we obtain:

\begin{lemma}
 \label{cont:lemma:ineq_norm}
 $\|\varphi(\uw_1)\|\,\|\varphi(\uw_2)\| \le \|\varphi(\uw_1\uw_2)\|
\le 2\|\varphi(\uw_1)\|\, \|\varphi(\uw_2)\|$ for any
$\uw_1,\uw_2\in\cW$.
\end{lemma}

\begin{proof}
This follows by observing that, for any non-empty word $\uw\in\cW$,
the matrix $\varphi(\uw)$ takes the form $\matrice{a}{b}{c}{d}$ with
$a\ge \max\{b,c\}$ and $\min\{b,c\}\ge d\ge 0$, and so
$\|\varphi(\uw)\| = a$.
\end{proof}

We say that an irrational real quadratic number $\alpha$ is
\emph{reduced} if $0<\alpha<1$ and $\alphabar<-1$ where $\alphabar$
denotes the conjugate of $\alpha$ over $\bQ$. Such a number is
characterized as follows:

\begin{lemma}
 \label{cont:lemma:reduced}
Let $\alpha$ be an irrational real quadratic number.  Then $\alpha$
is reduced if and only if its continued fraction expansion takes the
form $\alpha = [0,\Pi^\infty] = [0,\Pi,\Pi,\dots]$ for some
non-empty word $\Pi=(a_1,\dots,a_k)$ in $\cW$.  When this happens
the conjugate $\alphabar$ of $\alpha$ is given by $-\alphabar =
[(\Pi^*)^\infty]=[\Pi^*,\Pi^*,\dots]$ where $\Pi^* =
(a_k,\dots,a_1)$ is the reverse of $\Pi$.  Moreover we have
$\varphi(\Pi)\cdot(1/\alpha) = 1/\alpha$ and $H(\alpha) \le
\|\varphi(\Pi)\|$.

Conversely, if $0<\alpha<1$ and if
$\varphi(\Pi)\cdot(1/\alpha)=1/\alpha$ for some non-empty word
$\Pi\in\cW$, then $\alpha=[0,\Pi^\infty]$ and so $\alpha$ is
reduced.
\end{lemma}

\begin{proof}
The first two assertions are due to E.~Galois \cite{Ga}.  The other two follow
from the fact that the condition $\varphi(\Pi)\cdot(1/\alpha)=1/\alpha$ is
equivalent to $1/\alpha=[\Pi,1/\alpha]$, which is itself equivalent to
$\alpha=[0,\Pi^\infty]$, while a short computation shows that it implies
$H(\alpha)\le \|\varphi(\Pi)\|$.
\end{proof}

Since any extremal number comes with exactly two conjugates, it is
natural to transpose the notion of reduced irrational real quadratic
number to extremal numbers by stating:

\begin{definition}
 \label{cont:def:ext_red}
An extremal number $\xi$ is \emph{reduced} if $0<\xi<1$ and if its
conjugates $\xi'$ and $\xi''$ satisfy $\xi'<-1$ and $\xi''<-1$.
\end{definition}

\begin{lemma}
 \label{cont:lemma:ext_red}
Let $\xi = [a_0,a_1,a_2,\dots]$ be an extremal number in continued
fraction form.  For each sufficiently large index $i\ge 1$, the
number $\xi_i:=[0,a_i,a_{i+1},a_{i+2},\dots]$ is a reduced extremal
number in the $\GL_2(\bZ)$-equivalence class of $\xi$. Moreover, for
any $i\ge 1$ for which $\xi_i$ is reduced, the two conjugates of
$\xi_{i+1}$ belong to the open interval $(-a_i-1,-a_i)$.
\end{lemma}

\begin{proof}
Let $\xi'$ and $\xi''$ denote the conjugates of $\xi$.  By Corollary
\ref{conj:cor:A.xi}, each $\xi_i$ is extremal with conjugates
$\xi'_i$ and $\xi''_i$ given recursively by
\[
 \xi'_1=\xi'-a_0,
 \quad
 \xi''_1=\xi''-a_0,
 \quad
 \xi'_{i+1} = \frac{1}{\xi'_i}-a_{i},
 \quad
 \xi''_{i+1} = \frac{1}{\xi''_i}-a_{i}
 \quad
 (i\ge 1).
\]
Moreover, since $\xi'$ and $\xi''$ are distinct from $\xi$, they do
not have the same continued fraction expansion, and so have
$\xi'_i<-1$ and $\xi''_i<-1$ for each sufficiently large $i$. For
each of those $i$, the number $\xi_i$ is reduced.  The last
assertion is clear.
\end{proof}

In particular, each extremal number is equivalent to infinitely many
reduced ones.  We now show that this ambiguity disappears with the
following stronger notion.

\begin{definition}
An extremal number is \emph{balanced} if it is reduced and if its
conjugates have distinct integral parts.
\end{definition}

\begin{proposition}
 \label{cont:prop:balanced}
Any extremal number is equivalent to a unique balanced extremal
number.
\end{proposition}

\begin{proof}
\emph{Existence:} Let $\xi_1$ be an extremal number with conjugates
denoted $\xi'_1$ and $\xi''_1$.  In order to show that $\xi_1$ is
equivalent to a balanced extremal number, we may assume, in view of
Lemma \ref{cont:lemma:ext_red}, that it is reduced.  Then, we find
continued fraction expansions of the form
\[
   \xi_1  = [0, a_1, a_2, a_3, \dots], \quad
 -\xi'_1  = [a'_0, a'_{-1}, a'_{-2}, \dots], \quad
 -\xi''_1 = [a''_0, a''_{-1}, a''_{-2}, \dots],
\]
for sequences of positive integers $(a_i)_{i\ge 1}$, $(a'_i)_{i\le
0}$ and $(a''_i)_{i\le 0}$.  If $a'_0\neq a''_0$, then $\xi_1$ is
already balanced.  Otherwise, since $\xi'_1\neq \xi''_1$, there
exists a largest integer $k\le -1$ such that $a'_k\neq a''_k$. For
each $i=0,-1,\dots,k+1$, we put $a_i:=a'_i=a''_i$ and define
recursively $\xi_i := 1/(a_i+\xi_{i+1})$, $\xi'_i :=
1/(a_i+\xi'_{i+1})$ and $\xi''_i := 1/(a_i+\xi''_{i+1})$.  For each
of those $i$, we have
\[
   \xi_i  = [0, a_i, a_{i+1}, a_{i+2}, \dots], \quad
 -\xi'_i  = [a'_{i-1}, a'_{i-2},  \dots], \quad
 -\xi''_i = [a''_{i-1}, a''_{i-2}, \dots],
\]
and, by Corollary \ref{conj:cor:A.xi}, the number $\xi_i$ is
extremal with conjugates $\xi'_i$ and $\xi''_i$.  In particular,
$\xi$ is equivalent to $\xi_{k+1}$ which is balanced.

\emph{Uniqueness:} Let $\xi$ and $\eta$ be equivalent balanced
extremal numbers.  In order to complete the proof of the
proposition, it remains only to show that $\xi=\eta$.  To this end,
write $\xi=[0,a_1,a_2,\dots]$ and $\eta=[0,b_1,b_2,\dots]$.  Since
$\xi$ and $\eta$ are equivalent, it follows from Serret's theorem
\cite[Ch.~I, Thm.~6B]{Sc}, that there exist integers $k,\ell\ge 1$
such that $a_{k+i} = b_{\ell+i}$ for each $i\ge 0$.  Choose $k$
minimal with this property and define $\zeta = [0,a_k,a_{k+1},\dots]
= [0,b_\ell,b_{\ell+1},\dots]$.  If $k>1$, Lemma
\ref{cont:lemma:ext_red} shows that $\zeta$ has conjugates in the
interval $(-a_{k-1}-1, -a_{k-1})$.  Similarly, if $\ell>1$, it shows
that these conjugates lie in the interval $(-b_{\ell-1}-1,
-b_{\ell-1})$.  If $k>1$ and $\ell>1$, this means that
$a_{k-1}=a_{\ell-1}$, against the choice of $k$.  Thus, we must have
$k=1$ or $\ell=1$, and so $\zeta$ is equal to $\xi$ or $\eta$.  In
particular, $\zeta$ is balanced.  In view of the above, this is
possible only if $k=\ell=1$ which means that $\zeta=\xi=\eta$ as
requested.
\end{proof}

The following simple fact is the only combinatorial property that we
will need about the continued fraction expansion of general extremal
numbers.

\begin{proposition}
 \label{cont:prop:cube}
Let $\xi=[0,a_1,a_2,a_3,\dots]$ be the continued fraction expansion
of an extremal real number from the interval $(0,1)$.  There are
finitely many finite words $\Pi\in\cW$ whose cube is a prefix of $P
:= a_1a_2a_3\cdots$.
\end{proposition}

\begin{proof}
Suppose that $\Pi^3$ is a prefix of $P$ for some finite word
$\Pi\in\cW$, and consider the quadratic real number
$\alpha:=[0,\Pi^\infty]$.  By Lemma \ref{cont:lemma:reduced}, we
have $H(\alpha)\le \varphi(\Pi)$ and the theory of continued
fractions shows that
\[
 |\xi-\alpha| \le 2\, \|\varphi(\Pi^3)\|^{-2}.
\]
Thanks to Lemma \ref{cont:lemma:ineq_norm}, we deduce from this that
$|\xi-\alpha| \le 2\, \|\varphi(\Pi)\|^{-6} \le 2 H(\alpha)^{-6}$.
By Proposition \ref{conj:prop:alpha}, this holds only for finitely
many quadratic numbers $\alpha$.  In turn, this means that
$\|\varphi(\Pi)\|$ is bounded above and so $\Pi$ belongs to a finite
set of prefixes of $P$.
\end{proof}

We now turn to a characterization of the continued fraction
expansions of the extremal numbers $\xi_\um$. In view of the
formulas \eqref{ext:thm_xi:lim}, the first step is to describe the
continued fraction expansion of the quadratic numbers $\alpha_\um$.
For this, we denote by $\cW_0$ the sub-monoid of $\cW$ generated by
the words $\ua = (1,1) = 1\,1$ and $\ub = (2,2) = 2\,2$.  We let the
endomorphims of $\cW_0$ act on the right on $\cW_0$ and denote by
$U$ and $V$ the specific such endomorphisms determined by the
conditions
\begin{equation}
 \label{cont:endo_UV}
  \ua^U=\ua\ub,\quad \ub^U=\ub \et \ua^V=\ua, \quad \ub^V=\ua\ub.
\end{equation}
as in \cite[\S3]{Bo}.  Building on these, we form a tree of endomorphisms of
$\cW_0$:
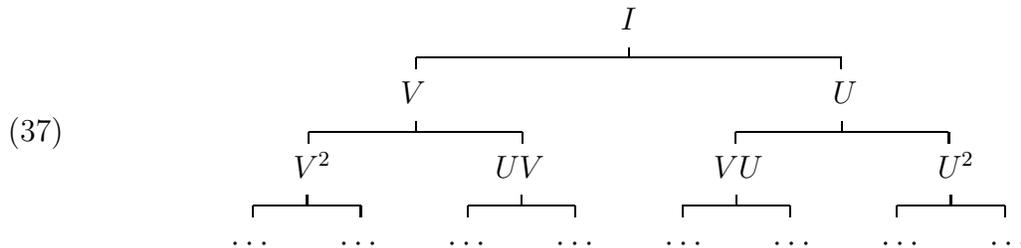
\begin{equation}
 \label{cont:tree:endo}
 \begin{matrix}
  I\\[-3pt]
  \acci\\[3pt]
  V  \hspace{154pt} U\\[-3pt]
  \accii \hspace{80pt} \accii\\[3pt]
  \hspace{4pt} V^2 \hspace{62pt} UV \hspace{64pt}
  VU \hspace{66pt} U^2\\[-3pt]
  \acciii \hspace{40pt} \acciii \hspace{40pt}
  \acciii \hspace{40pt} \acciii\\[3pt]
  \cdots \hspace{25pt} \cdots \hspace{25pt} \cdots \hspace{25pt}
  \cdots \hspace{25pt}
  \cdots \hspace{25pt} \cdots \hspace{25pt} \cdots \hspace{25pt}
  \cdots
  \end{matrix}
\end{equation}
where each node $\psi$ has successors $V\psi$ on the left and
$U\psi$ on the right.  For each node $\um$ of the Markoff tree
\eqref{M:Mtree}, we denote by $\psi_\um$ the endomorphism of $\cW_0$
which occupies the same position.  This gives for example
$\psi_{(5,1,2)}=I$ and $\psi_{(194,13,5)}=UV$.

\begin{lemma}
 \label{cont:lemma:Pi_m}
For each $\um\in\Sigma$, the quadratic number $\alpha_\um$ given by
\eqref{M:alpha_m} is reduced and its continued fraction expansion is
$\alpha_\um = [0,(\Pi_\um)^\infty]$ where $\Pi_\um = \ua$ if
$\um=(1,1,1)$, $\Pi_\um = \ub$ if $\um=(2,1,1)$ and $\Pi_\um =
(\ua\ub)^{\psi_\um}$ otherwise.
\end{lemma}

\begin{proof}
The formulas \eqref{M:alpha_m} show that each $\alpha_\um$ is a
reduced quadratic real number because, in the notation of
\eqref{M:alpha_m}, Proposition \ref{M:propStree} gives $1\le k\le
m\le 2k$. Moreover, since $F_\um(1,\alpha_\um)=0$, we find that
$(\ux_\um M) \cdot (1/\alpha_\um) = 1/\alpha_\um$. Thus, in view of
Lemma \ref{cont:lemma:reduced}, it remains simply to prove that
$\ux_\um M = \varphi(\Pi_\um)$ for each $\um\in\Sigma$. This is a
simple computation if $\um$ is one of the degenerate triples
$(1,1,1)$ or $(2,1,1)$. For the remaining triples, we claim more
precisely that the node $(\ux,\ux_1,\ux_2)$ of
\eqref{M:propStree:tree} which occupies the same position as $\um$
in the Markoff tree \eqref{M:Mtree} satisfies
\begin{equation}
 \label{cont:lemma:Pi_m:eq}
 \ux M = \varphi\big( (\ua\ub)^{\psi_\um} \big),
 \quad
 \ux_1 M = \varphi\big( \ua^{\psi_\um} \big)
 \quad
 \ux_2 M = \varphi\big( \ub^{\psi_\um} \big).
\end{equation}
Again, this is a quick computation for the root $(5,1,2)$ of the
Markoff tree because, for that triple, we have $\psi_\um=I$ and we
find
\[
 \ux M = \matrice{12}{5}{7}{3} = \varphi(\ua\ub),
 \quad
 \ux_1 M = \matrice{2}{1}{1}{1} = \varphi(\ua),
 \quad
 \ux_2 M = \matrice{5}{2}{2}{1} = \varphi(\ub).
\]
Assume that \eqref{cont:lemma:Pi_m:eq} holds for some node $\um$ of
the Markoff tree. The left successor of $(\ux,\ux_1,\ux_2)$ in
\eqref{M:propStree:tree} is $(\ux_1M\ux,\ux_1,\ux)$ and we find
\[
 \begin{aligned}
 \ux_1M\ux M
  &= \varphi\big( \ua^{\psi_\um} \big)
     \varphi\big( (\ua\ub)^{\psi_\um} \big)
   = \varphi\big( (\ua\ua\ub)^{\psi_\um} \big)
   = \varphi\big( (\ua\ub)^{V\psi_\um} \big),\\
 \quad
 \ux_1 M
  &= \varphi\big( \ua^{\psi_\um} \big)
   = \varphi\big( \ua^{V\psi_\um} \big),\\
 \quad
 \ux M
  &= \varphi\big( (\ua\ub)^{\psi_\um} \big)
   = \varphi\big( \ub^{V\psi_\um} \big),
 \end{aligned}
\]
where $V\psi_\um$ is the left successor of $\psi_\um$ in
\eqref{cont:tree:endo}.  Similarly, we find that
\eqref{cont:lemma:Pi_m:eq} holds with $(\ux,\ux_1,\ux_2)$ replaced
by its right successor $(\ux M\ux_2, \ux, \ux_2)$ and $\psi_\um$
replaced by its right successor $U\psi_\um$.  This proves our claim
by induction on the level of $\um$ and therefore completes the proof
of the lemma.
\end{proof}

\begin{theorem}
 \label{cont:thm}
Let $\xi=[0,a_1,a_2,a_3,\dots]$ denote the continued fraction
expansion of an irrational real number $\xi$ with $0<\xi<1$.  Then
$\xi$ belongs to the set $\{\xi_\um\,;\, \um\in\Sigma^*\}$ if and
only if there exists a finite product $\psi$ of\/ $U$ and\/ $V$ such
that $(\ua\ub)^{(VU)^i\psi}$ is a prefix of $P:=a_1a_2a_3\cdots$ for
each $i\ge 0$.
\end{theorem}

\begin{proof}
Suppose first that $\xi=\xi_\um$ for some $\um\in\Sigma^*$, and let
$(\um^{(i)})_{i\ge 1}$ denote the maximal zigzag in \eqref{M:Mtreex}
starting with $\um^{(1)}=\um$.  Define $\psi:=\psi_{\um^{(r)}}$
where $r=1$ if $m^{(2)}$ is the right successor of $\um$ and $r=2$
otherwise.  Then, for each $i\ge 0$, we have $\psi_{\um^{(2i+r)}} =
(VU)^i\psi$ and the above Lemma \ref{cont:lemma:Pi_m} gives
$\alpha_{\um^{(2i+r)}} = [0,(\Pi_i)^\infty]$ with $\Pi_i :=
(\ua\ub)^{(VU)^i\psi}$.  Since $\ua\ub$ is a prefix of
$(\ua\ub)^{VU} = \ua\ub\ua\ub\ub$, we note that $\Pi_i$ is a prefix
of $\Pi_{i+1}$ for each $i\ge 0$.  Combining this with the fact
that, by Theorem \ref{ext:thm_xi}, the sequence
$(\alpha_{\um^{(2i+r)}})_{i\ge 0}$ converges to $\xi_\um$, we deduce
that $\Pi_i$ must be a prefix of $P$ for each $i\ge 0$.

Conversely, suppose that there exists a finite product $\psi$ of $U$
and $V$ such that $\Pi_i := (\ua\ub)^{(VU)^i\psi}$ is a prefix of
$P$ for each $i\ge 0$.   For each $i\ge 1$, denote by $\um^{(2i-1)}$
and $\um^{(2i)}$ the nodes of the Markoff tree \eqref{M:Mtree} for
which $(VU)^{i-1}\psi = \psi_{\um^{(2i-1)}}$ and $U(VU)^{i-1}\psi =
\psi_{\um^{(2i)}}$.  Then, by Lemma \ref{cont:lemma:Pi_m}, we have
$\xi=\lim_{i\to\infty}\alpha_{\um^{(2i-1)}}$ and, by construction,
the sequence $(\um^{(i)})_{i\ge 1}$ is a zigzag in the tree
\eqref{M:Mtreex} with $\um^{(2)}$ as the right successor of
$\um^{(1)}$.  This zigzag is contained in maximal one starting with
some triple $\um\in\Sigma^*$.  As Theorem \ref{ext:thm_xi} shows
that $\xi_\um = \lim_{i\to \infty} \alpha_{\um^{(2i-1)}}$, we
conclude that $\xi=\xi_\um$.
\end{proof}

%
%

\section{Critical doubly infinite words}
 \label{sec:crit}

For each doubly infinite word $A = \cdots a_{-2} a_{-1} a_0 a_1 a_2
\cdots$ on the set of positive integers, we define
\begin{equation}
 \label{crit:L(A)}
 L(A) = \sup_{i\in\bZ}\big( [0, a_i, a_{i+1}, \dots] +
 [a_{i-1}, a_{i-2}, \dots] \big) \in [0,\infty].
\end{equation}
The relevance of this quantity to our problem is provided by the
following key formula for the infimum of reduced real indefinite
quadratic forms on $\bZ^2\setminus\{(0,0)\}$ (see \cite[Appendix
1]{CF} or \cite[pp.~80--81]{Di}):

\begin{proposition}
 \label{crit:propL}
Let $\xi,\eta$ be irrational real numbers with $0<\xi<1$ and
$\eta<-1$. Write
\[
 \xi = [0, a_1, a_2, a_3, \dots]
 \et
 -\eta = [a_0, a_{-1}, a_{-2}, \dots].
\]
Then the quadratic form $G(U,T) = (T-\xi U)(T- \eta U) \in \bR[U,T]$
has
\[
 \frac{\mu(G)}{\sqrt{\disc(G)}}
 = L(\cdots a_{-2} a_{-1} a_0 a_1 a_2 \cdots)^{-1}.
\]
\end{proposition}

Our goal in this ultimate section is to show that any extremal
number $\xi$ with Lagrange constant $\nu(\xi) = 1/3$ is equivalent
to $\xi_\um$ for some $\um\in\Sigma^*$.  In view of Proposition
\ref{cont:prop:balanced}, we may restrict to balanced extremal
numbers.  Then, by combining the above proposition with Theorem
\ref{red:thm}, we obtain the following statement.

\begin{corollary}
 \label{crit:propL:cor}
Let $\xi$ be a balanced extremal number with $\nu(\xi)=1/3$.  Denote
by $\xi'$ and $\xi''$ its conjugates and form the continued fraction
expansions
\[
 \xi=[0,a_1,a_2,a_3,\dots], \quad
 -\xi' = [a'_0, a'_{-1},a'_{-2},\dots] \et
 -\xi'' = [a''_0, a''_{-1},a''_{-2},\dots].
\]
Then, the semi-infinite words $P := a_1a_2a_3\cdots$, $Q' := \cdots
a'_{-2}a'_{-1}a'_0$ and $Q'' := \cdots a''_{-2}a''_{-1}a''_0$
satisfy $L(Q'P)=L(Q''P)=3$.  Moreover, $P$ is not ultimately
periodic and we have $a'_0\neq a''_0$.
\end{corollary}

\begin{proof}
Let $G'$ and $G''$ denote the real quadratic forms associated to
$\xi$ (see Definition \ref{conj:def}). According to Proposition
\ref{crit:propL}, we have
\[
 \frac{\mu(G')}{\sqrt{\disc(G')}} = L(Q'P)^{-1}
 \et
 \frac{\mu(G'')}{\sqrt{\disc(G'')}} = L(Q''P)^{-1}.
\]
Then Theorem \ref{red:thm} gives $L(Q'P) = L(Q''P) = \nu(\xi)^{-1} =
3$. Finally, $P$ is not ultimately periodic because $\xi$ is not a
quadratic number, and we have $a'_0\neq a''_0$ because $\xi$ is
balanced.
\end{proof}

In their presentation of Markoff's theory, both L.~E.~Dickson
\cite{Di} and E.~Bombieri \cite{Bo} provide a combinatorial analysis
of the doubly infinite words $A$ with $L(A)\le 3$.  Those with
$L(A)<3$ are well understood.  They are exactly the purely periodic
words with period $\ua$, $\ub$ or $(\ua\ub)^{\psi_\um}$ for some
$\um$ in the Markoff tree \eqref{M:Mtree} \cite[Thm.~15]{Bo}, and so
they form a countable set.  By contrast the doubly infinite words
$A$ with $L(A)=3$ make an uncountable set.  Among these, some are
\emph{ultimately periodic} in the sense that they admit a periodic
right semi-infinite suffix such as the word $1^\infty\, 2\, 2\,
1^\infty = \cdots\, 1\, 1\, 2\, 2\, 1\, 1\, \cdots$ (see
\cite[Thm.~63]{Di}). Putting these aside, we state:

\begin{definition}
A doubly infinite word $A$ is \emph{critical} if it has $L(A)=3$ and
is not ultimately periodic.
\end{definition}

In the context of Corollary \ref{crit:propL:cor}, we are facing two
critical words $Q'P$ and $Q''P$ with common suffix $P$.  Our next
goal is to provide a combinatorial analysis of this situation.
Collecting results from the presentation of Bombieri in \cite{Bo},
we first make the following observation.

\begin{lemma}
 \label{crit:lemmaC}
Let $A$ be a critical word.  There exist an integer $e\ge 1$ and a
non-constant sequence $(e_i)_{i\ge\bZ}$ consisting of integers from
the set $\{e,e+1\}$ such that $A$ factors as
\begin{equation}
\label{crit:lemmaC:eq}
 \cdots\ua\ub^{e_{-1}}\ua\ub^{e_0}\ua\ub^{e_1}\cdots\ \emph{(type I)}
 \quad\text{or}\quad
 \cdots\ub\ua^{e_{-1}}\ub\ua^{e_0}\ub\ua^{e_1}\cdots\ \emph{(type II).}
\end{equation}
Moreover, if $A$ is of type I (resp.\ type II), there exists a
unique doubly infinite product $B$ of the words $\ua$ and $\ub$ such
that $A = B^{U^e}$ (resp.\ $A = B^{V^e}$), and $B$ is critical of
type II (resp.\ type I) .
\end{lemma}

\begin{proof}
Since $A$ is not ultimately periodic, Lemma 11 of \cite{Bo} shows
that it can be written in one of the forms \eqref{crit:lemmaC:eq}
for some non-constant sequence of positive integers
$(e_i)_{i\ge\bZ}$.  Suppose that $A$ is of type I, and put $e =
\min_{i\in\bZ} e_i$.  Then, we have $A=B^{U^e}$ with $B =
\cdots\ua\ub^{e_{-1}-e}\ua\ub^{e_0-e}\ua\ub^{e_1-e}\cdots$. Like
$A$, this word $B$ is not ultimately periodic and Lemma 14 of
\cite{Bo} gives $L(A) = L(B)=3$, thus $B$ is a critical word.  Upon
choosing an index $i$ such that $e_i=e$, we find that $B$ contains
the subword $\ua\ub^{e_i-e}\ua = \ua\ua$, thus $B$ is of type II.
From this it follows that each difference $e_j-e$ is equal to $0$ or
$1$, thus $e_j\in\{e,e+1\}$.  The case where $A$ is of type II is
similar.
\end{proof}

The second preliminary result given below is connected to the fact
that, for each $\um$ in the Markoff tree \eqref{M:Mtree}, the
matrices $\ux_\um$ of Section \ref{sec:M} are symmetric and satisfy
$\ux_\um M = \varphi((\ua\ub)^{\psi_\um})$ (see the proof of Lemma
\ref{cont:lemma:Pi_m}).

\begin{lemma}
 \label{crit:lemmaP}
For any finite product $\psi$ of $U$ and $V$, the word $(\ua
\ub)^\psi$ admits a factorization of the form $\ua \up \ub$ where
$\up = \up^*$ is a palindrome in $\cW_0$.
\end{lemma}

The combinatorial argument given below is extracted from the proof
of Theorem 15 of \cite{Bo}.

\begin{proof}
We proceed by induction on the length of $\psi$ as a product of $U$
and $V$.  If this length is $0$, we have $(\ua\ub)^\psi = \ua\up\ub$
where $\up=\emptyset$ is the empty word.  Otherwise, $\psi$ takes
one of the forms $\psi'U$ or $\psi'V$ for some product $\psi'$ of
$U$ and $V$ of smaller length.  By hypothesis, we have
$(\ua\ub)^{\psi'} = \ua\up'\ub$ for some palindrome $\up'\in\cW_0$.
Then $(\ua\ub)^\psi$ is either equal to $(\ua\up'\ub)^U$ or
$(\ua\up' \ub)^V$ and so it takes the form $\ua\up\ub$ where $\up$
is either $\ub(\up')^U$ or $(\up')^V\ua$.  As $\up'$ is a
palindrome, the formulas (7) of \cite{Bo} show that, in both cases
$\up$ is a palindrome.
\end{proof}

\begin{theorem}
 \label{crit:thm}
Let $P = a_1a_2a_3\dots$ be a right semi-infinite word which is not
ultimately periodic.  The following conditions are equivalent:
\begin{itemize}
 \item[1)]
 There exist left semi-infinite words $Q'= \dots
 a'_{-2}a'_{-1}a'_{0}$ and $Q''= \dots  a''_{-2}a''_{-1}a''_{0}$
 with $a'_0\neq a''_0$ such that $L(Q'P)=L(Q''P)=3$.
 \item[2)]
 There exists a sequence of positive integers $(n_i)_{i\ge 1}$ such
 that, upon defining recursively
 \begin{equation}
 \label{crit:thm:def_psi}
 \psi_1=U^{n_1-1},\quad
 \psi_i =
 \begin{cases}
  V^{n_i}\psi_{i-1} &\text{if $i\ge 2$ is even,}\\
  U^{n_i}\psi_{i-1} &\text{if $i\ge 3$ is odd,}
  \end{cases}
 \end{equation}
 the word $\ua^{\psi_i}$ is a prefix of $\ua P$ for each $i\ge 1$.
\end{itemize}
Moreover, when the condition 1) is fulfilled, one of the words $Q'$
or $Q''$ is $P^*\ua\ub$ and the other is $P^*\ub\ua$, where $P^*$
denotes the reciprocal of $P$.
\end{theorem}

In the sequel, we only use the implication 1) $\Rightarrow$ 2).
However the reverse implication shows in particular that there are
uncountably many right semi-infinite words $P$ satisfying 1).

\begin{proof}
Suppose first that the condition 1) is fulfilled.  Then, the words
$A':=Q'P$ and $A'':=Q''P$ are both critical and, as they admit $P$
for suffix, they are products of $\ua$ and $\ub$ of the same type
(see Lemma \ref{crit:lemmaC}).  By permuting the words $Q'$ and
$Q''$ if necessary, we may assume without loss of generality that
$Q'$ ends with $1$ and that $Q''$ ends with $2$.

Suppose first that $A'$ and $A''$ are of type I.  Then there exist
sequences of positive integers $(e'_i)_{i\in\bZ}$ and
$(e''_i)_{i\in\bZ}$ such that
\[
 A' = \cdots\ua\ub^{e'_{-1}}\ua\ub^{e'_0}\ua\ub^{e'_1}\cdots
 \et
 A'' = \cdots\ua\ub^{e''_{-1}}\ua\ub^{e''_0}\ua\ub^{e''_1}\cdots.
\]
Since $A'$ and $A''$ admit $P$ as a common suffix, these two
sequences coincide from some point on.  By shifting the indexation,
we may assume that $e'_0\neq e''_0$ and that $e'_i = e''_i$ for each
$i\ge 1$.  As $P$ is not ultimately periodic, the integers
$e_i:=e'_i=e''_i$ with $i\ge 1$ are not all equal to each other.
Then, according to Lemma \ref{crit:lemmaC}, the sequences
$(e'_i)_{i\in\bZ}$, $(e''_i)_{i\in\bZ}$ and $(e_i)_{i\ge 1}$ take
values in the same set $\{e,e+1\}$ for some integer $e\ge 1$.  As
the suffix $P$ is preceded by $1$ in $A'$ and by $2$ in $A''$, we
deduce that $e'_0=e$ and $e''_0=e+1$, so that
\begin{equation}
 \label{crit:thm:eq1}
 Q'= \cdots \ua\ub^{e'_{-2}}\ua\ub^{e'_{-1}}\ua, \quad
 Q''= \cdots \ua\ub^{e''_{-2}}\ua\ub^{e''_{-1}}\ua\ub, \quad
 P =  \ub^e\ua\ub^{e_1}\ua\ub^{e_2}\cdots,
\end{equation}
and therefore
\[
 A' = (Q_1'P_1)^{U^e} \et A'' = (Q_1''P_1)^{U^e}
\]
for some left semi-infinite words $Q'_1$ with suffix $\ua$ and
$Q''_1$ with suffix $\ua\ub$, and some right semi-infinite word
$P_1$ such that
\begin{equation}
 \label{crit:thm:eq2}
 \ua P = (\ua P_1)^{U^e}.
\end{equation}
By Lemma \ref{crit:lemmaC}, the words $A_1':=Q'_1P_1$ and
$A_1'':=Q''_1P_1$ are both critical of type II.

As the suffix $P_1$ is preceded by $1$ in $A'_1$ and by $2$ in
$A''_1$, the same argument based on Lemma \ref{crit:lemmaC} shows
that there exist an integer $f\ge 1$ and sequences $(f'_i)_{i<0}$
$(f''_i)_{i<0}$ and $(f_i)_{i>0}$ taking values in $\{f,f+1\}$ such
that
\begin{equation}
 \label{crit:thm:eq3}
 Q'_1= \cdots \ub\ua^{f'_{-2}}\ub\ua^{f'_{-1}}\ub\ua, \quad
 Q''_1= \cdots \ub\ua^{f''_{-2}}\ub\ua^{f''_{-1}}\ub, \quad
 P_1 =  \ua^f\ub\ua^{f_1}\ub\ua^{f_2}\cdots.
\end{equation}
From this, we deduce that
\[
 A'_1 = (Q'_2P_2)^{V^f} \et A''_1 = (Q''_2P_2)^{V^f}
\]
for some left semi-infinite words $Q'_2$ with suffix $\ub\ua$ and
$Q''_2$ with suffix $\ub$, and some right semi-infinite word $P_2$
such that
\begin{equation}
 \label{crit:thm:eq4}
 \ua P_1 = \ua P_2^{V^f} = (\ua P_2)^{V^f}.
\end{equation}
Then, by Lemma \ref{crit:lemmaC}, the words $A_2':=Q'_2P_2$ and
$A_2'':=Q''_2P_2$ are both critical of type I.

Combining \eqref{crit:thm:eq2} and \eqref{crit:thm:eq4}, we obtain
\[
 \ua P = (\ua P_1)^{U^e} = (\ua P_2)^{V^fU^e}.
\]
Moreover, \eqref{crit:thm:eq1} and \eqref{crit:thm:eq3} show that
$\ub\ua$ is a suffix of $Q'$ and $Q'_1$ while $\ua\ub$ is a suffix
of $Q''$ and $Q''_1$.  Therefore, by iterating the above
construction indefinitely, we obtain a sequence of positive integers
$(n_i)_{i\ge 1}$ starting with $n_1=e+1$ and $n_2=f$, two sequences
of left semi-infinite words $(Q'_i)_{i\ge 1}$ and $(Q''_i)_{i\ge
1}$, and a sequence of right semi-infinite words $(P_i)_{i\ge 1}$
with the following properties. For each $i\ge 1$, the word $\ub\ua$
is a suffix of $Q'_i$, the word $\ua\ub$ is a suffix of $Q''_i$, and
we have
\begin{equation}
 \label{crit:thm:eq5}
 A' = (Q'_iP_i)^{\psi_i},
 \quad
 A'' = (Q''_iP_i)^{\psi_i}
 \et
 \ua P = (\ua P_i)^{\psi_i},
\end{equation}
for the sequence $(\psi_i)_{i\ge 1}$ defined by
\eqref{crit:thm:def_psi}. If $A'$ and $A''$ are of type II, we reach
the same conclusion upon starting with $n_1=1$, $Q'_1=Q'$,
$Q''_1=Q''$ and $P_1=P$. Then, in all cases, we deduce from the last
equality in \eqref{crit:thm:eq5} that $\ua^{\psi_i}$ is a prefix of
$\ua P$ for each $i\ge 1$, and this proves 2).

Lemma \ref{crit:lemmaP} shows that $(\ua\ub)^\psi = \ua^{U\psi} =
\ub^{V\psi}$ takes the form $\ua\up\ub$ with a palindrome
$\up\in\cW_0$ for any product $\psi$ of $U$ and $V$.  Thus, for any
integer $i\ge 1$, we can write
\[
 \ub^{\psi_{2i}} = \ua\up_{2i}\ub
 \et
 \ua^{\psi_{2i+1}} =  \ua\up_{2i+1}\ub
\]
for some palindromes $\up_{2i}$ and $\up_{2i+1}$.  Since
$\psi_{2i+1} = U^{n_{2i+1}}\psi_{2i}$, we find that
\[
 (\ua\ub)^{\psi_{2i+1}}
  = \ua^{\psi_{2i+1}} \ub^{\psi_{2i}}
  = \ua\up_{2i+1}\ub\ua\up_{2i}\ub.
\]
Thus $\up_{2i+1}\ub\ua\up_{2i}$ is a palindrome, and so
\begin{equation}
 \label{crit:thm:eq6}
 \up_{2i+1}\ub\ua\up_{2i} = \up_{2i}\ua\ub\up_{2i+1}.
\end{equation}
This shows in particular that $\up_{2i}$ is a prefix of $\up_{2i+1}$
because, since $\ub^{\psi_{2i}}$ is a proper suffix of
$\ua^{\psi_{2i+1}} = (\ua \ub^{n_{2i+1}})^{\psi_{2i}}$, the length
of $\up_{2i}$ as a product of $\ua$ and $\ub$ is shorter than that
of $\up_{2i+1}$.

Fix any index $i\ge 1$.  By \eqref{crit:thm:eq5}, we have $\ua P =
(\ua P_{2i+1})^{\psi_{2i+1}}$, thus
\begin{equation}
 \label{crit:thm:eq7}
 P = \up_{2i+1}\ub P_{2i+1}^{\psi_{2i+1}}.
\end{equation}
In particular, $\up_{2i+1}$ is a prefix of $P$ and so $\up_{2i}$ is
also a prefix of $P$.  Since $\ua\ub$ is a suffix of $Q''_{2i+1}$,
we deduce from \eqref{crit:thm:eq5} that $A''$ admits the suffix
\[
 \begin{aligned}
 (\ua\ub P_{2i+1})^{\psi_{2i+1}}
  &= \ua\up_{2i+1}\ub\ua\up_{2i}\ub P_{2i+1}^{\psi_{2i+1}} \\
  &= \ua\up_{2i}\ua\ub\up_{2i+1}\ub P_{2i+1}^{\psi_{2i+1}}
     \quad &\text{by \eqref{crit:thm:eq6},}\\
  &= \ua\up_{2i}\ua\ub P
     \quad &\text{by \eqref{crit:thm:eq7}.}
 \end{aligned}
\]
Thus, $\up_{2i}\ua\ub$ is a common suffix of $Q''$ and $P^*\ua\ub$.
Similarly, since $\ub\ua$ is a suffix of $Q'_{2i}$, the formulas
\eqref{crit:thm:eq5} show that $A'$ admits the suffix
\[
 (\ub\ua P_{2i})^{\psi_{2i}}
  = \ub^{\psi_{2i}} \ua P
  = \ua\up_{2i}\ub\ua P,
\]
thus, $\up_{2i}\ub\ua$ is a common suffix of $Q'$ and $P^*\ub\ua$.
Letting $i$ go to infinity, we deduce that $Q''=P^*\ua\ub$ and that
$Q' = P^*\ub\ua$.

Conversely, assume that $P$ satisfies the condition 2) of the
theorem.  To complete the proof, it remains only to show that
$L(P^*\ua\ub P) = L(P^*\ub\ua P) =3$.  Since $P^*\ua\ub P$ is the
reverse of $P^*\ub\ua P$, Lemma 5 of \cite{Bo} reduces this task to
showing that $L(P^*\ub\ua P) =3$.  Since the palindrome $\up_{2i+1}$
is a prefix of $P$ whose length goes to infinity with $i$, any
finite subword of $P^*\ub\ua P$ is contained in $\up_{2i+1} \ub \ua
\up_{2i+1}$ for some $i\ge 1$, and so is contained in the purely
periodic word $\cdots\Pi_{2i+1}\Pi_{2i+1}\Pi_{2i+1}\cdots$ with
period $\Pi_{2i+1} = \ua^{\psi_{2i+1}} = \ua\up_{2i+1}\ub$.  By
Theorem 15 of \cite{Bo} this word has $L(\cdots \Pi_{2i+1}
\Pi_{2i+1} \cdots) < 3$ (because $\Pi_{2i+1} = (\ua\ub)^\psi$ with
$\psi = U^{n_{2i+1}-1}\psi_{2i}$).  By continuity, this implies that
$L(P^*\ub\ua P) \le 3$.  Since $P$ is not ultimately periodic, this
must be an equality \cite[Thm.~15]{Bo}.
\end{proof}

We can now complete the proof of our main result which reads as
follows.

\begin{theorem}
 \label{crit:thm:main}
The set $\{\xi_\um\,;\, \um\in\Sigma^*\}$ constitute a system of
representatives of the $\GL_2(\bZ)$-equivalence classes of extremal
numbers $\xi$ with $\nu(\xi)=1/3$.
\end{theorem}

\begin{proof}
According to Theorem \ref{ext:thm_xi} the extremal numbers $\xi_\um$
with $\um\in\Sigma^*$ are two by two inequivalent and, by Corollary
\ref{red:thm:cor2}, their Lagrange constant is $1/3$.  It remains to
show that any extremal number $\xi$ with $\nu(\xi)=1/3$ is
equivalent to one of these. As mentioned at the beginning of this
section, in order to show this, we may assume, by Proposition
\ref{cont:prop:balanced}, that $\xi$ is balanced. Then Corollary
\ref{crit:propL:cor} shows that its continued fraction expansion
takes the form $\xi=[0,P]$ where $P$ is a right semi-infinite word
on positive integers which is not ultimately periodic and satisfies
the condition 1) of Theorem \ref{crit:thm}. Let $(n_i)_{i\ge 1}$ be
the sequence of positive integers such that, for the corresponding
sequence $(\psi_i)_{i\ge 1}$ of endomorphisms of $\cW_0$ given by
\eqref{crit:thm:def_psi}, the word $\ua^{\psi_i}$ is a prefix of
$\ua P$ for each $i\ge 1$. Define
\[
 \uv_i = \begin{cases}
       \ua^{\psi_i} &\text{if $i\ge 1$ is odd,}\\
       \ub^{\psi_i} &\text{if $i\ge 2$ is even.}
       \end{cases}
\]
The recurrence relations \eqref{crit:thm:def_psi} translate into
\begin{align}
 \label{crit:thm:main:eq1}
 \uv_{2i+1}
  &= \ua^{\psi_{2i+1}}
   = (\ua\ub^{n_{2i+1}})^{\psi_{2i}}
   = \uv_{2i-1}\uv_{2i}^{n_{2i+1}}, \\
 \label{crit:thm:main:eq2}
 \uv_{2i+2}
  &= \ub^{\psi_{2i+2}}
   = (\ua^{n_{2i+2}}\ub)^{\psi_{2i+1}}
   = \uv_{2i+1}^{n_{2i+2}}\uv_{2i}.
\end{align}

We know that $\uv_{2i+1}$ is a prefix of $\ua P$ for each $i\ge 1$.
We claim that the reverse $\uv_{2i}^*$ of $\uv_{2i}$ is a prefix of
$\ub P$ for each $i\ge 1$. To prove this, we note, as in the proof
of Theorem \ref{crit:thm}, that $\uv_{2i+1}$ is the images of
$\ua\ub$ by $U^{n_{2i+1}-1}\psi_{2i}$ and so, by Lemma
\ref{crit:lemmaP}, it takes the form $\uv_{2i+1} = \ua\up_{2i+1}\ub$
for some palindrome $\up_{2i+1}$.  Then, $\up_{2i+1}$ is a prefix of
$P$.  Moreover, the formula \eqref{crit:thm:main:eq1} implies that
$\uv_{2i}$ is a suffix of $\up_{2i+1}\ub$.  Thus, $\uv_{2i}^*$ is a
prefix of $\ub\up_{2i+1}$ and so is a prefix of $\ub P$.

Using \eqref{crit:thm:main:eq1} and \eqref{crit:thm:main:eq2}, we
also note that, for each $i\ge 2$, the word
\[
 \uv_{2i+1}
  = \uv_{2i-1}\uv_{2i}^{n_{2i+1}}
  = \uv_{2i-1} (\uv_{2i-1}^{n_{2i}}\uv_{2i-2})^{n_{2i+1}}
\]
admits $\uv_{2i-1}^{n_{2i}+1}$ as a prefix, while the word
\[
 \uv_{2i}^*
  = (\uv_{2i-1}^{n_{2i}}\uv_{2i-2})^*
  = \big( (\uv_{2i-3}\uv_{2i-2}^{n_{2i-1}})^{n_{2i}} \uv_{2i-2}
  \big)^*
\]
admits $(\uv_{2i-2}^*)^{n_{2i-1}+1}$ as a prefix.  Therefore,
$\uv_{2i-1}^{n_{2i}+1}$ is a prefix of $\ua P$ and
$(\uv_{2i-2}^*)^{n_{2i-1}+1}$ is a prefix of $\ub P$ for each $i\ge
2$.  Since $[0,\ua P]$ and $[0,\ub P]$ are the continued fraction
expansions of fixed extremal numbers (in the equivalence class of
$\xi$), we deduce from Proposition \ref{cont:prop:cube} that
$n_{2i}=n_{2i+1}=1$ for each sufficiently large integer $i$, say for
$i\ge i_0$.  Then, upon putting $\psi_0=\psi_{2i_0}$, we obtain
\[
 \psi_{2i+1} = U(VU)^{i-i_0}\psi_0
\]
for each $i\ge i_0$, and so
\[
 \ua^{\psi_{2i+1}} = (\ua\ub)^{(VU)^{i-i_0}\psi_0}
\]
is a prefix of $\ua P$ for each $i\ge i_0$.  By Theorem
\ref{cont:thm} this implies that $[0,\ua P] = \xi_\um$ for some
$\um\in\Sigma^*$.
\end{proof}

We conclude with the following result which provides an additional
link between extremal numbers and Markoff's theory.

\begin{corollary}
 \label{crit:thm:main:cor}
Let $\xi$ be an extremal number and let $(\alpha_i)_{i\ge 1}$ be a
sequence of best quadratic approximations to $\xi$ in the sense of
Definition \ref{conj:def:best_app}.  Then the following assertions
are equivalent:
\begin{itemize}
 \item[1)] $\nu(\xi) = 1/3$,
 \item[2)] $\nu(\alpha_i) > 1/3$ for each sufficiently large $i$,
 \item[2)] $\nu(\alpha_i) > 1/3$ for infinitely many $i$.
\end{itemize}
\end{corollary}

\begin{proof}
Suppose first that $\nu(\xi)=1/3$.  Then, by the preceding theorem,
$\xi$ is equivalent to $\xi_\um$ for some $\um\in\Sigma^*$ and so,
by Lemma \ref{conj:lemma:best_app}, each $\alpha_i$ with $i$
sufficiently large is equivalent to $\alpha_\un$ or $\alphabar_\un$
for some $\un\in\Sigma^*$.  According to Markoff's Theorem
\ref{M:thmM}, these quadratic numbers have $\nu(\alpha_\un) =
\nu(\alphabar_\un) >1/3$.  This means that $\nu(\alpha_i)>1/3$ for
each sufficiently large $i$, and a fortiori for infinitely many
values of $i$.

Conversely, suppose that $\nu(\alpha_{i_j})> 1/3$ for a strictly
increasing sequence of positive integers $(i_j)_{j\ge 1}$.  Without
loss of generality, we may assume that these integers $i_j$ all have
the same parity. Then, by Proposition \ref{conj:prop:xi'xi''}, the
sequence $(\alphabar_{i_j})_{j\ge 1}$ converges to some conjugate
$\xi'$ of $\xi$ and so, upon defining
\[
 F_j(U,T):= (T - \alpha_{i_j} U) (T - \alphabar_{i_j} U)
 \et
 G'(U,T):=(T-\xi U)(T-\xi'U),
\]
we obtain $G'(U,T)/\sqrt{\disc(G')} = \lim_{j\to\infty}
F_j(U,T)/\sqrt{\disc(F_j)}$, thus
\[
 \nu(\xi) = \frac{\mu(G')}{\sqrt{\disc(G')}} \ge \limsup_{j\to\infty}
\frac{\mu(F_j)}{\sqrt{\disc(F_j)}}\cdot
\]
where the first equality comes from Theorem \ref{red:thm}.  By
Markoff's Theorem \ref{M:thmM}, the above limit superior is equal to
$1/3$.  This gives $\nu(\xi)\ge 1/3$ and we conclude that
$\nu(\xi)=1/3$ since $\xi$ is not a quadratic number.
\end{proof}

\begin{lastremark}
For each $\xi\in\bR$, denote by $\lambdahat_2(\xi)$ the supremum of
all real numbers $\lambda>0$ such that the inequalities $|x_0|\le
X$, $|x_0\xi-x_1|\le X^{-\lambda}$ and $|x_0\xi^2-x_2|\le
X^{-\lambda}$ admit a non-zero solution $(x_0,x_1,x_2)\in\bZ^3$ for
each sufficiently large value of $X$.  By \cite{Rexp}, we know that
the values taken by $\lambdahat_2$ on the set of non-quadratic
irrational real numbers are dense in the interval $[1/2,1/\gamma]$.
It would be interesting to know what happens if instead we consider
the values taken by $\lambdahat_2$ on the set of irrational numbers
$\xi$ with $\nu(\xi)=1/3$.  By looking at Sturmian continued
fractions, Y.~Bugeaud and M.~Laurent showed in \cite[Thm.~3.1]{BL}
that, for each bounded sequence of positive integers $(s_i)_{i\ge
1}$, there exists a real number $\xi$ with $\lambdahat_2(\xi) =
(1+\sigma)/(2+\sigma)$ where $\sigma = \liminf_{k\to\infty}
[0,s_k,s_{k-1},,\dots,s_1]$. I think that, by considering
appropriate paths in the Markoff tree \eqref{M:Mtreex} like in
\S\ref{sec:ext}, one should be able to produce real numbers $\xi$
with the same exponents $\lambdahat_2$ and with $\nu(\xi)=1/3$.  By
analogy with work of S.~Fischler in \cite{Fi}, it is possible that
this exhausts the set of all possible values taken by $\lambdahat_2$
on the real numbers $\xi$ with $\nu(\xi)=1/3$.

\end{lastremark}

\end{document}